\documentclass[12pt]{amsart}
\usepackage{amsmath, amsthm, amssymb, graphicx, url, xspace, blkarray}
\usepackage[font=normalsize,labelfont=bf]{caption}

\usepackage{pgf,tikz}
\usetikzlibrary{arrows}

\usepackage{paralist}

\usepackage[square,numbers]{natbib}
\usepackage{geometry}
\newcommand{\txn}[1]{\textnormal{#1}}
\newcommand{\rsphere}{\mathbb{C}\mathbb{P}^1}
\newcommand{\algbr}{\overline{\mathbb{Q}}}
\newcommand{\absgal}{\txn{Gal}(\algbr/\mathbb{Q})}
\newcommand{\inv}[1]{#1^{-1}}

\newcommand{\per}{(\sigma,\alpha,\varphi)}
\newcommand{\pergrp}{\left\langle\sigma,\alpha,\varphi\right\rangle}
\newcommand{\pergrpj}{\left\langle\sigma c_j, \alpha, c_j\varphi\right\rangle}

\newcommand*{\symmdiff}{\bigtriangleup}
\newcommand{\belyi}{Bely\u{\i}\xspace}
\newcommand{\q}{\mathbb Q}
\newcommand{\dtroid}[1]{\Delta (#1)}

\newcommand{\form}[2]{\left\langle #1,#2\right\rangle}
\newcommand{\mform}[2]{\left(#1|#2\right)}
\newcommand{\sprod}[2]{\left[#1,#2\right]}
\newcommand{\orth}{\q^E\oplus\q^{E^*}}
\newcommand{\map}{\mathcal M}
\newcommand{\zoi}{\{0,1,\infty\}}
\newcommand{\h}{H_1(S\setminus\{V\cup V^*\})}

\newtheorem{theorem}{Theorem}[section]
\newtheorem*{theorem*}{Theorem}
\theoremstyle{definition}
\newtheorem{lemma}{Lemma}[section]

\newtheorem{example}{Example}[section]

\newtheorem{constr}{Construction}[section]

\newtheorem{definition}{Definition}[section]
\newtheorem{rmrk}{Remark}[section]

\title[Action of $BC_n$ on maps, matroids and representations]{An action of the Coxeter group $BC_n$ on maps on surfaces, Lagrangian matroids and their representations}
\author{Goran Mali\'{c}}
\address{School of Mathematics, University of Manchester, Oxford Road, Manchester M13 9PL, United Kingdom}
\email{goranm00@gmail.com}

\begin{document}
\bibliographystyle{plainnat}

\begin{abstract}
For a map $\mathcal M$ cellularly embedded on a connected and closed orientable surface, the bases of its Lagrangian (also known as delta-) matroid $\dtroid\map$ correspond to the bases of a Lagrangian subspace $L$ of the standard orthogonal space $\orth$, where $E$ and $E^*$ are the edge-sets of $\mathcal M$ and its dual map. The Lagrangian subspace $L$ is said to be a representation of both $\mathcal M$ and $\dtroid\map$. Furthermore, the bases of $\dtroid\map$, when understood as vertices of the hypercube $[-1,1]^n$, induce a polytope $\mathbf P(\dtroid\map)$ with edges parallel to the root system of type $BC_n$. In this paper we study the action of the Coxeter group $BC_n$ on $\mathcal M$, $L$, $\dtroid\map$ and $\mathbf P(\dtroid\map)$. We also comment on the action of $BC_n$ on $\map$ when $\map$ is understood a dessin d'enfant.

\smallskip
\noindent \textbf{Keywords.} Matroids, matroid polytopes, delta-matroids, partial duals, hyperoctahedral group, dessins d'enfants.

\smallskip
\noindent \textbf{MSC2010.} Primary: 05B35. Secondary: 20F55, 11G32, 14H57, 17B22, 57M60.

\end{abstract}

\maketitle

\section{Introduction}

It is well known that planar connected graphs can be characterised by their graphic matroids: a connected graph $G$ is planar if, and only if its graphic matroid $M(G)$ is co-graphic \cite{oxley92}, i.e.\ if there is a connected graph $H$ such that $M(H)$ is isomorphic to the dual matroid $M(G)^*$ of $M(G)$. Furthermore, if $G$ is planar, then it has a geometric dual $G^*$, and the two matroids $M(G)^*$ and $M(G^*)$, are isomorphic. Non-planar graphs cellularly embedded on orientable, connected and closed surfaces, otherwise known as \emph{maps on surfaces}, satisfy a similar relation if we consider \emph{delta-matroids} instead. In that case, if $i(G)$ is a cellular embedding of $G$ into a orientable, connected and closed surface, then its delta-matroid $\dtroid {i(G)}$ is isomorphic to the delta-matroid $\dtroid {i(G)^*}$ of the geometric dual of $i(G)$, and planar graphs can be characterised as those graphs $G$ for which $\dtroid {i(G)}$ is a matroid \cite{moffat14}.

Delta-matroids were introduced by Bouchet in \cite{bouchet87} and are related to maps on surfaces in the same way as matroids are related to graphs: the collection of bases of a delta-matroid is in a 1-1 correspondence with the collection of bases of a map (see section \ref{section:maps} for the definition of a base of a map). 

Delta-matroids are also known as Lagrangian matroids \cite{borovik_gelfand_white} for the fact that they capture the underlying combinatorics of Lagrangian subspaces sitting inside symplectic $2n$-spaces, similarly to how matroids capture the combinatorial essence of linear (in)de\-pen\-dence in vector spaces. Throughout the paper we shall use the name Lagrangian matroid in order to emphasise the connection with Lagrangian subspaces.

Lagrangian matroids arising from maps on surfaces have the additional property of being representable: if $\mathbb F$ is a field of characteristic 0 and $\map$ a map with $n$ edges, then to the Lagrangian matroid $\dtroid\map$ we can assign an $n$-dimensional Lagrangian subspace $L$ of the standard orthogonal $2n$-space $\mathbb F^n\oplus \mathbb F^n$, such that the bases of $L$ correspond to the bases of $\dtroid\map$. In that case, $\dtroid\map$ is represented by a $k\times 2n$ matrix $\mform XY$, where $X$ and $Y$ are $k\times n$ matrices over $\mathbb F$ with $XY^t$ symmetric. The aim of this paper is to study such representations.

The representation of $\dtroid\map$ is best understood if we study it together with the representation of the delta-matroid $\dtroid {\partial_j \map}$ corresponding to the map $\partial_j \map$ called \emph{the partial dual of} $\map$ (with respect to the edge $j$). In his 2009 paper \cite{chmutov09}, Chmutov introduced the operation of \emph{generalised duality}, now known as \emph{partial duality}, which generalises the geometric dual of a map; it is an edge-by-edge operation which for a map $\map$, and some edge $j$ of $\map$, produces a map $\partial_j \map$ with the same number of edges as $\map$, but not necessarily the same number of vertices or faces. After dualising every edge of $\map$, the geometric dual $\map^*$ of $\map$ is restored.

It was noted in \cite{moffat14} that the delta-matroids of $\map$ and its partial duals $\partial_j \map$ have compatible structure, i.e.\ there is a 1-1 correspondence between the bases of $\map$ and the bases of $\partial_j \map$, and furthermore the symmetric difference operator $A\symmdiff B=(A\cup B)\setminus(A\cap B)$ establishes a 1--1 correspondence between the bases of their respective Lagrangian matroids: for each base $B'$ in $\dtroid{\partial_j \map}$, there is a unique base $B$ in $\dtroid \map$ such that $B'=B\symmdiff\{j,j^*\}$, where $j^*$ denotes the dual edge of $j$. In section \ref{section: representation of partial duals} we show that the representations behave accordingly:
\begin{quote}
the representation $L_j$ of $\dtroid{\partial_j \map}$ can be obtained from the representation $L$ of $\dtroid \map$ by interchanging the column-vectors $j$ and $j^*$.
\end{quote}

This operation can also be thought of as an action of a transposition $(j~j^*)$ in the Coxeter group $BC_n$ on the columns of $L$. In fact, this action can be lifted to an action on maps with $n$ edges by defining $(j~j^*)\map$ to be the partial dual $\partial_j\map$. Since $BC_n$ is generated by such transpositions, together with the elements $(j~k)(j^*~k^*)$, which act on $\map$ by relabelling its edge-set, we have an action of $BC_n$ on maps. Hence, the aim of this paper is to study the action of $BC_n$ on: Lagrangian matroids, maps on surfaces and their representations as Lagrangian subspaces of the standard symplectic space.

\subsection*{Structure of the paper and main results}

This paper is organised as follows: in section \ref{section: lagrangian matroids} we start with the definition of Lagrangian matroids, and briefly discuss how they arise from totally isotropic subspaces of a standard symplectic or orthogonal space. We continue the section by discussing how to obtain Lagrangian matroids, and their representations from maps on surfaces.

In section \ref{section: partial duals} we introduce the cartographic group of a map, and define the partial duals of a map via its cartographic group. We recall the result of \cite{moffat14} and provide an alternative proof that the operation of partially dualising a map $\map$ with respect to a subset $A$ of its edge-set is compatible with taking the symmetric difference of its Lagrangian matroid $\dtroid\map$ in the sense that
\[\dtroid{\partial_A\map}=\dtroid\map\symmdiff (A\cup A^*),\]
where $A^*$ denotes the corresponding set of dual edges of the dual map $\map^*$. We also prove that the bases of a map $\map$ and the bases of its Lagrangian matroid $\dtroid\map$ are in a 1--1 correspondence with the partial duals of $\map$ which have exactly one face.

In section \ref{section: representation of partial duals} we discuss the representations of Lagrangian matroids of partial duals and the action of the Coxeter group $BC_n$ on maps on surfaces and Lagrangian matroids. We show that the representations $L_A$ of Lagrangian matroids of partial duals $\partial_A\map$ of a map $\map$ are obtained by transposing columns indexed by $A$ in the representation $L$ of the Lagrangian matroid of $\map$, which we can also understand as an action of $BC_n$ by transpositions $(j~j^*)$ on the representation $L$. We also show how acting by an element of $BC_n$ on a map affects its cartographic group.

Next, in section \ref{section:polytopes} we introduce the Lagrangian matroid polytope and discuss the action of $BC_n$ in terms of the root system of type $C_n$. In section \ref{section:pairs} we briefly touch upon the topic of pairs of Lagrangian subspaces and show that such pairs naturally arise from maps on surfaces and their partial duals with respect to a single edge.

In section \ref{section: application} we interpret maps as dessins d'enfants, i.e.\ algebraic curves $X$ over $\algbr$ together with a holomorphic ramified covering $X\to\rsphere$ ramified at most over a subset of $\zoi$. We show that when a map $\map$ is understood as a dessin d'enfant, then the partial duals of $\map$ with respect to some base of $\map$ can be defined over their field of moduli. Through an example we study how the natural Galois action on dessins d'enfants affects the fields of definition and the cartography groups of partial duals.

Finally, in section \ref{section:concluding} we provide some concluding remarks and provide some direction for future research.

\section{Lagrangian matroids associated to maps on surfaces, and representations}\label{section: lagrangian matroids}

\subsection{Lagrangian matroids}
Let $[n]=\{1,2,\dots,n\}$ and $[n]^*=\{1^*,2^*,\dots,n^*\}$ and introduce the maps
$*\colon [n]\to[n]^*$ with $j\mapsto j^*$, and $*\colon [n]^*\to[n]$ with $j^*\mapsto j$, so that $j^{**}=j$ and $*$ represents an involutive operation on $[n]\cup[n]^*$.

We say that an $n$-subset $A\subset[n]\cup[n]^*$ is admissible if $A\cap A^*=\emptyset$, or equivalently, if for all $j\in[n]$ precisely one of $j$ or $j^*$ appear in it.

Denote by $J_n$ the set of all admissible $n$-subsets of $[n]\cup[n]^*$, and let $\symmdiff$ be the symmetric difference operator, i.e.\ for sets $A$ and $B$ we have
\[A\symmdiff B=(A\cup B)\setminus(A\cap B).\]
\begin{definition}\label{def:symm}Let $\mathcal B$ be a collection of subsets in $J_n$. We say that the triple $([n]\cup [n]^*,*,\mathcal B)$ is a \emph{Lagrangian matroid} if it satisfies the \emph{symmetric exchange axiom}:
\begin{quote}for any $A,B\in\mathcal B$, and $j\in A\symmdiff B$, there exists $k\in B\symmdiff A$ such that $A\symmdiff\{j,k,j^*,k^*\}\in\mathcal B$.
\end{quote}
The elements in the collection $\mathcal B$ are the \emph{bases} of $([n]\cup [n]^*,*,\mathcal B)$.
\end{definition}
The symmetric exchange axiom obviously mimics the basis exchange axiom for matroids, and it is an easy exercise for the reader to show that every matroid is a Lagrangian matroid.

\begin{rmrk}The reason that we first write $A\symmdiff B$ and then $B\symmdiff A$ in definition \ref{def:symm}, even though the two sets are one and the same, is to emphasise the similarity with the basis exchange axiom for matroids, where the set difference operator plays the role of the symmetric difference operator.
\end{rmrk}

Two Lagrangian matroids $([m]\cup [m]^*,*,\mathcal B_1)$ and $([n]\cup [n]^*,*,\mathcal B_2)$ are \emph{isomorphic} if $m=n$ and there are bijections $(f,f^*)$ of $[n]$ and $[n]^*$, respectively, such that $f$ identifies the elements of $\mathcal B_1\cap[n]$ with those of $\mathcal B_2\cap[n]$, and $f^*$ identifies the elements of $\mathcal B_1\cap[n]^*$ with those of $\mathcal B_2\cap[n]^*$.

\begin{rmrk}It is enough to specify a bijection $f\colon[n]\to[n]$ identifying $\mathcal B_1\cap[n]$ with $\mathcal B_2\cap[n]$: if $f(j)=i_j$ is such a bijection, then define $f^*$ as $f^*(j^*)=i_j^*$. If $B_1$ is a base in $\mathcal B_1$ with $k$ unstarred elements, for simplicity
\[B_1=\{1,\dots,k,(k+1)^*,\dots,n^*\},\]
then there exists $B_2\in\mathcal B_2$ such that $f(B_1\cap[n])=B_2\cap [n]$ with
\[B_2=\{i_1,\dots,i_k,i_{k+1}^*,\dots,i_n^*\}.\]
It follows that $f^*(B_1\cap[n]^*)=\{i_{k+1}^*,\dots,i_n^*\}=B_2\cap[n]^*$.
\end{rmrk}

The most important examples of Lagrangian matroids are those that arise from maximal isotropic subspaces of a standard symplectic space. Such Lagrangian matroids are called \emph{representable}.

\subsubsection{Isotropic subspaces}

Let $V$ be a standard symplectic space, i.e.\ a $2n$-di\-men\-si\-o\-nal vector space over a characteristic 0 field $\mathbb F$ with a basis
\[E=\{e_1,e_2,\dots,e_n,e_{1^*},e_{2^*},\dots,e_{n^*}\},\]
and equipped with an anti-symmetric bilinear form $\form\cdot\cdot$ such that $\form{e_i}{e_j}=0$ whenever $i\neq j^*$, and $\form{e_i}{e_{i^*}}=1=-\form{e_{i^*}}{e_i}$ for all $i\in[n]$.

\begin{definition}We say that a subspace $W$ of $V$ is an \emph{isotropic} subspace of $V$ if the form $\form\cdot\cdot$ vanishes on $W$. The isotropic subspaces of maximum dimension are said to be \emph{Lagrangian}.\end{definition}

If $U$ is a $k$-dimensional isotropic subspace, then we must have
\[2n=\dim U+\dim{U^\perp}\geq k+k\]
since $U^\perp$ contains $U$ as a subspace. Therefore, the dimension of $U$ can be at most $n$, and $U$ is Lagrangian if, and only if $\dim U=n$.

\subsubsection{Representations of isotropic subspaces}

Let $U$ be a $k$-dimensional isotropic subspace of $V$ and $\{u_1,\dots,u_k\}$ a basis for $U$. We can expand the vectors $u_j$ in terms of the basis $E$ so that
\[u_i=\sum_{j=1}^n x_{ij}e_j+\sum_{j=1}^n y_{ij}e_{j^*},\]
and write down the coefficients $x_{ij}$ and $y_{ij}$ into a $k\times 2n$ matrix $\mform XY$ such that $X=(x_{ij})$ and $Y=(y_{ij})$. The matrix $\mform XY$ is said to be \emph{a representation} of $U$, and such matrices completely characterise isotropic subspaces of $V$.

\begin{lemma}A subspace $U$ of the standard symplectic space $V$ is isotropic if, and only if $U$ is represented by a matrix $\mform XY$ with $XY^t$ symmetric.\end{lemma}

We omit the proof for brevity, however a detailed discussion can be found in \cite[pp.~63]{borovik_gelfand_white}.

\subsection{Representable Lagrangian matroids}

It is clear that an $n\times 2n$ matrix $\mform XY$ will represent a Lagrangian subspace if, and only if the dimension of the row-space of $\mform XY$ is $n$ and $XY^t$ is an $n\times n$ symmetric matrix. Given such a matrix, let the columns of $X$ and $Y$ be indexed by the sets $[n]$ and $[n]^*$, respectively, and call an admissible $n$-subset $B$ of $[n]\cup[n]^*$ a base of $\mform XY$ if the columns indexed by $B$ form a non-zero $n\times n$ minor. Finally, let $\mathcal B$ be the collection of all bases of $\mform XY$.

\begin{theorem}If $U$ is a Lagrangian subspace, then $\mathcal B$ is the collection of bases of a Lagrangian matroid.\end{theorem}
\begin{proof}For $B=\{i_1,\dots,i_n\}\in \mathcal B$, denote again with $B$ the corresponding $n\times n$ sub-matrix of $\mform XY$. If $(X|Y)$ represents $U$, then elementary row-operations on $\mform XY$ leave both $U$ and the collection $\mathcal B$ invariant.

Assume first that $B_1=\{1,2,\dots,n\}\in \mathcal B$ so that $U$ is represented by $\mform{I_n}Y$, with $Y=Y^t$. Suppose that $B_2\in\mathcal B$ agrees with $B_1$ on $n-k$ columns. For clarity we shall assume that $B_2$ agrees with $B_1$ on the last $n-k$ columns so that $B_2=\{1^*,\dots,k^*,k+1,\dots,n\}$, since the general case differs only in the change of appropriate indices.

If $\mathcal B$ is not a Lagrangian matroid, then there exists $i\in\{1,1^*,\dots,k,k^*\}$ such that for all $j\in\{1,1^*,\dots,k,k^*\}$ the determinants $\det(B_1\symmdiff\{i,i^*,j,j^*\}\})$ are all zero. We may assume that $i=1$ since the same argument applies to the $i\neq 1$ case.

Let $Y=(y_{pq})$ with $y_{pq}=y_{qp}$. That $\det(B_1\symmdiff\{1,1^*\})$ is zero has $y_{11}=0$ as a consequence. Similarly,  $\det(B_1\symmdiff\{1,1^*,j,j^*\})=0$ implies that $y_{1j}y_{j1}=0$, which in turn implies that $y_{1j}=0$ for all $j=1,\dots,n$. Therefore, the first $k$ elements in the top row of $Y$, which are also the first $k$ elements in the top row of $B_2$, are all zero. Moreover, $B_2$ agrees with $I_n$ on the last $n-k$ rows, hence it has a zero row. Therefore, $\det B_2=0$, a contradiction, since $B_2\in\mathcal B$.

Now, if $B_1=\{i_1,i_2,\dots,i_n\}\in\mathcal B$ we may preform row operations on the representation $\mform XY$ of $U$ so to obtain a representation $\mform{X'}{Y'}$ in which $B_1$ is the identity matrix. In order to completely reduce the general case to the simplified one from above, note that $\mform{X'}{Y'}$, with $X'=(x_{pq})$ and $Y'=(y_{pq})$ is symmetric means that for any two rows $r$ and $s$ we have
\[(x_{r1}y_{s1}-x_{s1}y_{r1})+(x_{r2}y_{s2}-x_{s2}y_{r2})
+\cdots+(x_{rn}y_{sn}-x_{sn}y_{rn})=0.\]\end{proof}

Lagrangian matroids arising from Lagrangian subspaces of a symplectic space are called \emph{(symplectically) representable}.

\begin{example}Let $\mform XY$ be the following matrix with $XY^t$ symmetric:
\[\left(\begin{array}{ccc|ccc}
1 & 1 & 1 & 0 & 0 & 0\\
0 & 0 & 0 & 1 & 1 & -2\\
2 & 0 & 1 & -2 & 1 & 1
\end{array}\right).\]
The bases of the corresponding representable Lagrangian matroid are
\[\mathcal B=\{123^*,12^*3,12^*3^*,1^*23,1^*23^*,1^*2^*3\}.\]
\end{example}

\subsection{Orthogonal Lagrangian matroids}Let us consider now $V$ as a standard orthogonal space, i.e.\ a $2n$-di\-men\-si\-o\-nal space endowed with a symmetric bilinear form $\form\cdot\cdot$ such that $\form{e_i}{e_j}=0$ whenever $i\neq j^*$, and $\mform{e_i}{e_{i^*}}=\mform{e_{i^*}}{e_i}=1$. In the same fashion as before we may consider a $k$-dimensional isotropic subspace $U$, and represent it by a matrix $\mform XY$. However, in this case a matrix $\mform XY$ will correspond to an isotropic subspace if, and only if $XY^t$ is anti-symmetric \cite[pp.~77]{borovik_gelfand_white}.

If $\mform XY$ represents a Lagrangian subspace, we may obtain a Lagrangian matroid via the same construction as in the symplectic case, and for such Lagrangian matroids we say that they are \emph{orthogonally} representable.

\begin{example}Let $\mform XY$ be the following matrix with $XY^t$ anti-symmetric:
\[\left(\begin{array}{ccc|ccc}
0 & 1 & 1 & 1 & 0 & 0\\
-1 & 0 & 0 & 0 & 1 & 0\\
-1 & 0 & 0 & 0 & 0 & 1
\end{array}\right).\]
The bases of the corresponding Lagrangian matroid are
\[\mathcal B=\{123^*,12^*3,1^*2^*3^*\}.\]
\end{example}

\subsection{Maps on surfaces and Lagrangian matroids}

Let $G$ be a connected graph, possibly with loops and multiple edges. By a surface we shall understand a connected, closed and orientable surface of genus $g\geq 0$. We say that an embedding $i(G)$ of $G$ into a surface $S$ is cellular if the vertices of $G$ are points on $S$, the edges are segments on $S$ homeomorphic to closed $1$-cells which intersect only at the vertices, and $S\setminus i(G)$ is a disjoint union of connected components called \emph{faces}, each homeomorphic to an open $2$-cell.

Any connected graph can be cellularly embedded on some surface \cite{lindsay59}, and for such an embedding we shall say that it is a \emph{map on a surface}. In some literature the faces of a map are also called \emph{countries}, and a distinguished point within a country is called a \emph{capitol}. Two maps $\mathcal M_1$ and $\mathcal M_2$ on a surface $S$ are said to be isomorphic if there is an orientation-preserving homeomorphism $S\to S$ which restricts to a graph-isomorphism between the underlying graphs of $\mathcal M_1$ and $\mathcal M_2$.

\begin{constr}\emph{Dual map}. Let $\mathcal M$ be a map on $S$ with $E=\{e_1,\dots,e_n\}$ as its edge set. A dual map of $\mathcal M$ is a map on $S$ obtained by the following construction: for each face $f$ of $M$ choose a distinguished point in it and call it the \emph{face-centre} of $f$. If $e_j$ is on the boundary of two faces $f_1$ and $f_2$, define $e_j^*$ to be a segment (1-cell) connecting the face-centres of $f_1$ and $f_2$ passing through only $e_j$, exactly once. If $e_j$ is on the boundary of a single face $f$, then connect the face-centre to itself by a segment $e_j^*$ passing only through $e_j$, exactly once. The map with the face-centres of $\mathcal M$ as its vertex-set and the segments $E^*=\{e_1^*,\dots,e_n^*\}$ as its edge set is called \emph{a dual map} of $\mathcal M$.
\end{constr}

\begin{rmrk}The construction of a dual map of $\map$ is not rigid; we can obtain many dual maps of $\mathcal M$ since we have an uncountable amount of choices for the face-centres and the intersection points of $e_j$ and $e_j^*$. However, any two dual maps of $\mathcal M$ are isomorphic. We shall abuse this fact in order to speak of \emph{the} dual map of $\mathcal M$, denoted by $\mathcal M^*$.
\end{rmrk}

The vertices and edges of $\mathcal M^*$ are also called covertices and coedges, respectively, and we understand $e_j^*$ and $v_j^*$ to be the unique coedge and covertex of $e_j$ and $v_j$, respectively. The sets of vertices and covertices are denoted by $V$ and $V^*$, while the sets of edges and coedges are denoted by $E$ and $E^*$, respectively.

\subsection{Lagrangian matroids from maps on surfaces}\label{section:maps}

Let $\mathcal M$ be a map on $S$ with $n$ edges and $E$ and $E^*$ as its set of edges and coedges.

\begin{definition}A \emph{base} of $\mathcal M$ is an admissible $n$-subset $B$ of $E\cup E^*$ such that $S\setminus B$ is connected.\end{definition}

If $\mathcal M$ is a map on the sphere, i.e.\ if it is a plane graph, then every spanning tree of $\mathcal M$ is clearly contained in a unique base, and conversely, every base contains a unique spanning tree. If $\mathcal B$ is the collection of all bases of $\mathcal M$, then the collection
\[E\cap \mathcal B=\{E\cap B\mid B\in\mathcal B\}\]
corresponds to the set of spanning trees of the underlying graph of $\map$, i.e.\ to its graphic matroid. This need not be the case for maps on genus $g>0$ surfaces: the only spanning tree of the figure 8 graph is the trivial tree with one vertex and no edges, yet the figure 8 is its own base when embedded on a genus 1 surface. However, the collection $\mathcal B$ forms a Lagrangian orthogonal matroid, which shall be denoted by $\dtroid \map$.

\begin{theorem}\label{theorem:maps}Let $\mathcal M$ be a map on a surface $S$. The set $\mathcal B$ of its bases forms a Lagrangian matroid orthogonally representable over $\q$.\end{theorem}
\begin{proof}[Proof (sketch)]
The proof involves several parts, and its outline is as follows: first we introduce a symmetric bilinear form on $\q^E\cup \q^{E^*}$ so that $E\cup E^*$ forms a standard basis for it. Next, for a cycle $c\in H_1(S\setminus\{V\cup V^*\})$ we construct the incidence vector $\iota(c)\in\orth$ of $c$. Finally, we consider the image of $H_1(S\setminus\{V\cup V^*\})$ in $\orth$ under the map $c\mapsto \iota(c)$; it will turn out that this image is a Lagrangian subspace of $\orth$ whose bases correspond to the bases of $\mathcal M$.

We shall omit most details of the proof, however the the interested reader may look them up in Theorem 4.3.1 and section 4.3.2 in \cite{borovik_gelfand_white}.\vspace{\topsep}

Let $S$ have the positive, counter-clockwise orientation. Orient the edges of $\mathcal M$ arbitrarily, and orient the edges of $\mathcal M^*$ so that after a counter-clockwise turn, the orientation of $e$ coincides with the orientation of its coedge $e^*$. Equivalently, orient the edges of $\mathcal M^*$ so that $e^*$ intersects $e$ \emph{from the right}, as in figure \ref{figure: intersection index}. In that case, we also say that the \emph{intersection index} $(e,e^*)$ is equal to 1, and that $(e,e^*)=-1$ if $e^*$ intersects $e$ from the left.
\begin{figure}[ht]
  \centering
  \includegraphics[trim={5cm 16cm 5cm 7cm},scale=.8]{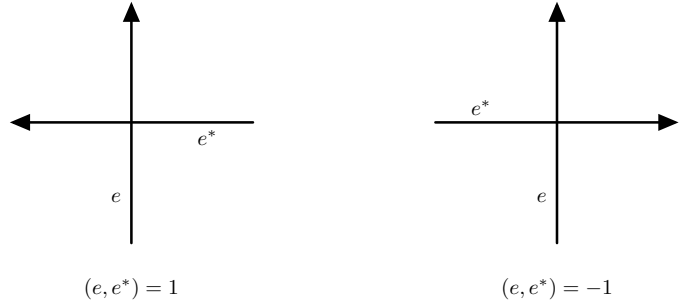}\\
  \caption{The intersection index $(e,e^*)$.}
  \label{figure: intersection index}
\end{figure}

The intersection index can be extended to the first homology group of $S$, and $S\setminus\{V\cup V^*\}$ in the obvious way. We shall use homology with coefficients in $\q$ so that $H_1(S)$ and $H_1(S\setminus\{V\cup V^*\})$ are vector spaces.\vspace{\topsep}

Define a symmetric bilinear form $\form\cdot\cdot$ on $\orth$ such that
\[\form{e}{e^*}=1,\quad\txn{and}\quad\form ef=0\txn{ whenever }f\neq e^*.\]
Hence $E\cup E^*$ is a standard basis for $\orth$ as an orthogonal space over $\q$.\vspace{\topsep}

For $c\in H_1(S\setminus\{V\cup V^*\})$ let
\[\iota(c)=\sum_{e\in E\cup E^*}(c,e)e\]
denote the \emph{incidence vector} of $c$ in $\orth$. Our aim is to show that the image of $H_1(S\setminus\{V\cup V^*\})$ under the map $c\mapsto \iota(c)$ is a Lagrangian subspace of $\orth$, with bases corresponding to the bases of $\mathcal M$.\vspace{\topsep}

Define a scalar product $\sprod\cdot\cdot$ on $H_1(S\setminus\{V\cup V^*\})$ with
\[\sprod cd=\form{\iota(c)}{\iota(d)}.\]
That $\iota(H_1(S\setminus\{V\cup V^*\}))$ is an isotropic subspace of $\orth$ will follow if we show that the scalar product $\sprod\cdot\cdot$ vanishes on $H_1(S\setminus\{V\cup V^*\})$. This is accomplished by considering the kernel $K$ of the canonical projection
\[H_1(S\setminus\{V\cup V^*\})\to H_1(S),\]
and noting that $\sprod{k}{c}=0$ for all $k\in K$ and $c\in H_1(S\setminus\{V\cup V^*\})$.\vspace{\topsep}

We now induce a scalar product on $H_1(S)$ denoted again by $\sprod\cdot\cdot$ defined by
\[\sprod cd=\sprod {c'}{d'},\txn{ for all }c,d\in H_1(S),\]
where $c'$ and $d'$ are any preimages of $c$ and $d$ in $H_1(S\setminus\{V\cup V^*\})$ under the canonical projection; it will be well defined since the scalar product on $H_1(S\setminus\{V\cup V^*\})$ vanishes on the kernel $K$.

Lifting the scalar product $\sprod\cdot\cdot$ to $H_1(S)$ allows us to proceed with the proof by induction: let $|V\cup V^*|=2$. Euler's formula implies that $\map$ has $2g$ edges, and $2g$ coedges, where $g$ is the genus of $S$. The group $\h$ is generated by two cycles $k_1$ and $k_2$ around the vertex and the covertex of $\map$, respectively, and clearly we have $\iota(k_1)=\iota(k_2)=0$. Hence $K$ is contained in the kernel $\ker\iota$, so we can lift $\iota$ to a map $H_1(S)\to\orth$.

The vector space $H_1(S)$ is $2g$-dimensional with $E$ and $E^*$ as its bases. Moreover, the intersection index
\[H_1(S)\times H_1(S)\to H_0(S)=\q\]
is a skew-symmetric form on $H_1(S)$, so $E$ and $E^*$ are actually dual bases. Hence, a cycle $c\in H_1(S)$ can be expressed as the sum $\sum c_ie_i$ for some coefficients $c_i$, and therefore the incidence vector $\iota(c)$ can be expressed as the sum $\sum c_i\iota(e_i)$. Since
\[\iota(e_i)=\sum_{j\in[n]\cup[n^*]}(e_i,e_j)e_j=\sum_{j\in [n]}\left((e_i,e_j)e_j+(e_i,e_j^*)e_j^*\right)\]
we can see that $[\iota(e_i),\iota(e_j)]=0$ by expanding, and taking into account that $(e_i,e_j^*)=0$, whenever $i\neq j^*$. Therefore, $\iota(\h)$ is a $2g$-dimensional isotropic subspace of $\orth$. Since $\orth$ is $4g$-dimensional, $\iota(\h)$ is Lagrangian.\vspace{\topsep}

The induction step follows from the fact that a map with $|V\cup V^*|>2$ vertices and covertices must have a contractible edge, and that contracting a contractible edge does not have any effect on the scalar product $\sprod\cdot\cdot$ on $\h$. This result is stated as lemma 4.3.4 in \cite{borovik_gelfand_white}, and we omit the details. Induction shows that the subspace $\iota(\h)$ is isotropic; we still have to show that it is Lagrangian, and that its bases correspond to the bases of $\map$.\vspace{\topsep}

Let $M$ be the infinite matrix with $[n]\cup[n]^*$ columns, and the elements of $\iota(\h)$ as its rows. Let $B$ be a basis of $\map$. By definition, the space $S\setminus B$ is connected, so for each edge or coedge $e$ in $B$ we can draw a cycle $c_e$ which intersects $e$ exactly once, without intersecting other edges or coedges in $B$.  Now consider the $n\times n$ sub-matrix of $M$ with columns indexed by $B$, and rows corresponding to $c_e$, for $e\in B$. By relabelling, we can rearrange this matrix so to get the identity matrix $I_n$. Hence, the $n\times n$ minor corresponding to $B$ has non-zero determinant, and the columns indexed by $B$ are linearly independent.

Conversely, let $B$ be a set of $n$ linearly independent columns in $M$. Since the $n$ columns in $B$ are linearly independent, we can find $n$ rows, each row corresponding to a cycle in $\h$, so that the $n\times n$ sub-matrix obtained in this way has non-zero determinant. Furthermore, by summing the cycles (as homotopy classes), and possibly after relabelling, we can choose this sub-matrix to be the identity matrix $I_n$. By definition of the incidence vector, each of the $n$ chosen cycles intersects only one edge or coedge in $B$, exactly once, and each edge or coedge in $B$ is intersected by exactly one of the $n$ cycles. Therefore, $S\setminus B$ must be connected.

Since the sub-matrix of $M$ corresponding to a base $B$ has rank $n$, and $\orth$ has dimension $|[n]\cup[n]^*|=2n$, it represents a Lagrangian subspace of $\orth$, and hence $\iota(\h)$ is a Lagrangian subspace of $\orth$.\end{proof}

\begin{example}\label{example: representation}Let $\mathcal M$ be the map on a genus 1 surface as shown and labelled in figure \ref{figure: map example 1}. The collection of its bases is
\[\mathcal B=\{123^*4^*,12^*34^*,1^*2^*3^*4^*\}.\]

\begin{figure}[ht]
  \centering
  \includegraphics[trim={5cm 14cm 5cm 5cm},scale=.6]{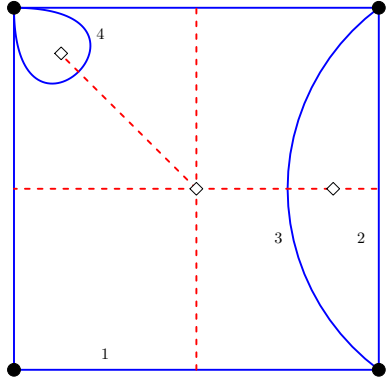}\\
  \caption{The map $\mathcal M$ from example \ref{example: representation}, together with its dual map $\mathcal M^*$ (dashed).}
  \label{figure: map example 1}
\end{figure}

To represent $\mathcal B$ as a Lagrangian subspace of $\orth$ we have to pick a base, say $B=123^*4^*$, and find cycles $c_i\in H_1(S\setminus\{V\cup V^*\})$ for $i\in B$ such that $c_i$ doesn't intersect the edges $e_j$, for all $j\in B\setminus \{i\}$. We may choose them as in figure \ref{figure: map representation example 1}.

\begin{figure}[ht]
  \centering
  \includegraphics[trim={5cm 14cm 5cm 5cm},scale=.6]{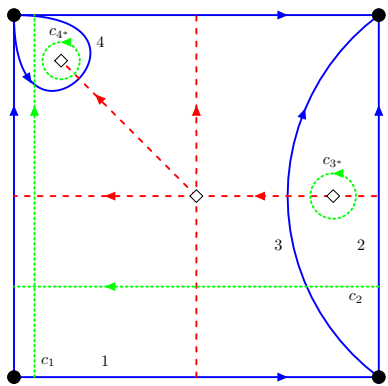}\\
  \caption{The cycles $c_i$ for $B=123^*4^*$ (dotted).}
  \label{figure: map representation example 1}
\end{figure}

We can now form the $4\times 2\cdot 4$ matrix with rows and columns indexed by $B$ and $[4]\cup[4]^*$, respectively. The entry $(i~j)$ will be the intersection index $\form{j}{c_i}$. A representation of $\mathcal B$ is given by the following anti-symmetric matrix:
\[
\begin{blockarray}{lcccccccc}
 \empty & 1 & 2 & 3 & 4 & 1^* & 2^* & 3^* & 4^*\\
 \begin{block}{l(cccc|cccc)}
    c_1 & 1 & 0 & 0 & 0 & 0 & 0 & -1 & 0\\
    c_2 & 0 & 1 & 1 & 0 & 1 & 0 & 0 & 0\\
    c_{3^*} & 0 & 0 & 0 & 0 & 0 & -1 & 1 & 0\\
    c_{4^*} & 0 & 0 & 0 & 0 & 0 & 0 & 0 & 1\\
 \end{block}
\end{blockarray}\ .
\]
The reader may check that the only admissible non-zero $n\times n$ minors are the ones indexed by $\mathcal B$.
\end{example}

\section{Partial duals}\label{section: partial duals}

The partial dual of a map $\mathcal M$ with respect to some subset $A$ of its edge-set is a map $\partial_A\mathcal M$ which can be thought of as an intermediate step between $\mathcal M$ and its dual map $\mathcal M^*$. The operation of partial duality was introduced by Chmutov in \cite{chmutov09}. We shall work with the algebraic description of partial duals in terms of the cartographic group of $\mathcal M$ due to Chmutov and Vignes-Tourneret \cite{chmutov14}.

\subsection{The cartographic group}Let $\mathcal M$ be a map on $S$, and let us subdivide each edge of $\mathcal M$ into two \emph{half-edges} by adding a distinguished vertex called \emph{the edge midpoint}. Thus, each edge of $\mathcal M$ has two half-edges, and two half-edges of an edge are incident to the same vertex if, and only if the edge is a loop.

A map $\mathcal M$ with $n$ edges induces a triple of permutations $\sigma$, $\alpha$ and $\varphi$ acting on the half-edges in the following way: let us label the half-edges of $\mathcal M$ with the elements of the set $[2n]$ so that when standing at a vertex, and looking towards an adjacent edge midpoint, the label is placed on the left. The triple of permutations $\sigma,\alpha,\varphi\in S_{2n}$ is defined as follows:
\begin{itemize}
  \item the disjoint cycles of $\sigma$ correspond to the counter-clockwise cyclic orderings of labels around each vertex;
  \item the disjoint cycles of $\alpha$ correspond to the counter-clockwise cyclic orderings of labels around each edge midpoint;
  \item the disjoint cycles of $\varphi$ correspond to the counter-clockwise cyclic orderings of labels within each face.
\end{itemize}
Consequently, $\alpha$ is a fixed-point free involution, the product $\sigma\alpha\varphi$ is trivial, and the group generated by $\sigma$, $\alpha$ and $\varphi$ acts transitively on the set $[2n]$.

\begin{definition}The subgroup $\left\langle\sigma,\alpha,\varphi\right\rangle$ of $S_{2n}$ generated by $\sigma$, $\alpha$ and $\varphi$ is called \emph{the cartographic group} of $\mathcal M$.\end{definition}

\begin{example}Figure \ref{figure:example} shows a genus 0 map with two vertices, edges and faces. The mark $\times$ denotes the midpoint of an edge. The permutations $\sigma$, $\alpha$ and $\phi$ are
\[\sigma=(1)(2~3~4), \alpha=(1~2)(3~4), \varphi=(1~4~2)(3),\]
and the cartographic group $\langle\sigma,\alpha,\varphi\rangle$ is the alternating group $A_4$ on 4 letters.

\begin{figure}[ht]
\centering
\includegraphics[trim={5cm 19.5cm 5cm 4.5cm}]{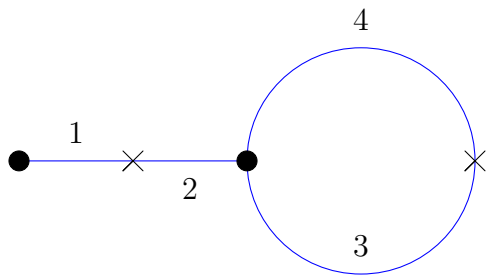}
\caption{A genus 0 map with 2 vertices, edges and faces. The half-edges are labelled so that when standing at a vertex, and looking towards an adjacent edge midpoint (represented by $\times$), the label is placed on the left.}
\label{figure:example}
\end{figure}
\end{example}

\begin{rmrk}The cartographic group is actually defined up to conjugation, since a relabelling of the half-edges of $\mathcal M$ corresponds to the simultaneous conjugation of $\sigma$, $\alpha$ and $\varphi$ by some element of $S_{2n}$. However, permutation groups are conjugate if, and only if they are isomorphic, so we shall talk of \emph{the} cartographic group.\end{rmrk}

Conversely, given a 2-generated subgroup $\langle x,y\rangle$ of $S_{2n}$ acting transitively on the set $[2n]$, and such that $y$ is a fixed-point free involution, an up to isomorphism unique map $\mathcal M$ can be recovered. The permutations $x$ and $y$ are responsible for reconstructing the vertices and edges, whilst the permutation $\inv{(xy)}$ keeps track of the gluing instructions for the faces. Hence, isomorphism classes of maps on surfaces with $n$ edges are in a one-to-one correspondence with the conjugacy classes of 2-generated transitive subgroups of $S_{2n}$.\vspace{\topsep}

\begin{rmrk}The cartographic group is a special instance of a \emph{monodromy group}. In general, a monodromy group is a 2-generated subgroup $\langle\sigma,\alpha\rangle$ of $S_n$ acting transitively on $[n]$ with no restrictions on the permutation $\alpha$, i.e.\ it doesn't necessarily have to be a fixed-point free involution. If $\alpha$ has a fixed point, then $\map$ will have a half-edge which doesn't glue glue into any edge. Furthermore, if $\alpha$ is not an involution, then there might be several half-edges that glue together and form a \emph{hyper-edge} (an edge that is incident with three or more vertices).  
\end{rmrk}

The genus of the surface of embedding can be computed directly from the cartographic group: clearly, the disjoint cycles of $\sigma$, $\alpha$ and $\varphi$ correspond bijectively to the vertices, edges and faces of $\mathcal M$. Therefore the Euler-Poincar\'e formula determines the genus:
\[2-2g=n(\sigma)-n(\alpha)+n(\varphi),\]
where $n(g)$ stands for the number of disjoint cycles of $g\in S_{2n}$.

\subsection{Duality and partial duals}It is not difficult to verify that if a map $\mathcal M$ corresponds to $\pergrp$, then the dual map $\mathcal M^*$ corresponds to $\left\langle\inv\varphi,\inv\alpha,\inv\sigma\right\rangle$.

\begin{definition}Let $\mathcal M = \pergrp$ be a map with $n$ edges, and let $\alpha=c_1c_2\cdots c_n$ be given as a product of disjoint transpositions so that $c_j$ corresponds to the edge $j$. The \emph{partial dual with respect to an edge} $j$ \emph{of} $\mathcal M$ is the map
\[\partial_j \mathcal M = \pergrpj.\]
\end{definition}

\begin{theorem}\label{theorem: partial}Partial dual with respect to an edge $j$ of a map $\mathcal M$ is well defined, i.e.\ $\sigma c_j\alpha c_j \varphi=1$ and the group $\pergrpj$ acts transitively on $\{1,\dots,2n\}$.\end{theorem}
\begin{proof}Since $c_j$ commutes with $\alpha$ we clearly have $\sigma c_j\alpha c_j \varphi=1$. Now let $a,b\in\{1,\dots,2n\}$, and suppose $c_j$ corresponds to an edge $j$ in $\mathcal M$ which is not a loop, and consider the map $\mathcal M\setminus j$ obtained from $\mathcal M$ by deleting $j$. If $\mathcal M\setminus j$ is connected, and $g\in\per$ is such $a^g=b$, then replacing every occurrence of $\sigma$ in $g$ with $\sigma c_j$ will not have any effect on how the half-edge $a$ moves through the map $\mathcal M\setminus j$ to get to $b$. Hence if we denote by $\hat g$ such an element in $\pergrpj$, we will have $a^{\hat g}=b$.

If $\mathcal M\setminus j$ has $\hat{\mathcal M}$ and $\tilde{\mathcal{M}}$ as its two connected components with $a\in\hat{\mathcal M}$ and $b\in\tilde{\mathcal{M}}$, then we can map them to the half-edges $d_1\in\hat{\mathcal M}$ and $d_2\in\tilde{\mathcal{M}}$ incident to the vertices of $e$ by some permutations $g_1$ and $g_2$ which again do not depend on $c_j$. Moreover, $d_1$ and $d_2$ can be chosen such that $d_1^{\sigma c_j}=d_2$. Therefore, if $\hat g_1$ and $\hat g_2$ are the permutations in $\pergrpj$ obtained from $g_1$ and $g_2$ by replacing every occurrence of $\sigma$ with $\sigma c_j$ we get
\[a^{\hat g_1}=d_1,\quad d_1^{\sigma c_j}=d_2,\quad b^{\hat g_2}=d_2,\]
from which $a^{\hat g_1 \sigma c_j \inv{\hat g_2}}=b$ follows.

The case in which $e$ is a loop is reduced to the previous case by first deleting $e$ from $\mathcal M$, and then considering the possibly disconnected map obtained from $\mathcal M\setminus j$ by splitting the vertex that was incident to $e$.\end{proof}

To \emph{see} how $\map$ transforms to $\partial_j\map$, consider first an edge $j$ of $\map$ which is not a loop. Let $a$ and $b$ be the half-edges of $j$, so that $v_a$ and $v_b$ are disjoint cycles in $\sigma$ with $a\in v_a$ and $b \in v_b$. To get the map corresponding to $\partial_j\map$, first contract $j$, and then attach a loop with half-edges labelled by $a$ and $b$ to the new vertex, so that the order of half-edges at the new vertex corresponds to $v_av_b(a~b)$, as in \mbox{figure \ref{figure: interpretation}}.

If $j$ is a loop, the transformation is reversed: first we delete $j$, and then expand its base vertex into a new edge $j$ with $a$ and $b$ as its half-edges so that the cyclic orderings of half-edges around the new vertices correspond to the two cycles of  $v_a(a~b)=v_b(a~b)$.

\begin{figure}[ht]
  \centering
  \includegraphics[trim={10cm 11.5cm 5cm 5.5cm},scale=.6]{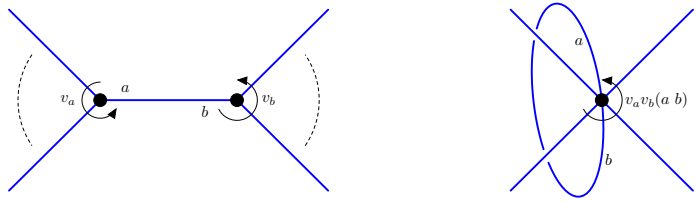}\\
  \caption{Transforming $\map$ to $\partial_j M$. First we contract, and then attach a loop so that the order of half-edges corresponds to $v_av_b(a~b)$.}\label{figure: interpretation}
\end{figure}

\begin{definition}Let $\mathcal M = \pergrp$ be a map with $n$ edges, $E$ its edge set, and let $\alpha=c_1c_2\cdots c_n$ be given as a product of disjoint transpositions, so that $c_j$ corresponds to the edge $j$. Let $A=\{i_1,\dots,i_k\}$ be some subset of $E$. The \emph{partial dual of $\mathcal{M}$ with respect to a subset $A\subseteq E$} is the map
\[\partial_A \mathcal M = \left\langle\sigma c_{i_1}\cdots c_{i_k},\alpha,c_{i_1}\cdots c_{i_k}\varphi\right\rangle.\]\end{definition}
That the partial dual $\partial_A\mathcal M$ is well defined follows directly from theorem \ref{theorem: partial} by an easy inductive argument.\vspace{\topsep}

The following lemma from \cite{chmutov09,chmutov14} lists some properties of the operation of partial duality, which are seen to be true directly from the definition.
\begin{lemma}\label{lemma: partial dual properties}Let $\mathcal M$ be a map, $E$ its set of edges, and $A$ some subset of $E$. Then
\begin{enumerate}[(a)]
  \item  $\partial_E \mathcal M = \mathcal M^*$
  \item  $\partial_A \partial_A \mathcal M= \mathcal M$.
  \item\label{lemma: partial dual properties: 3}  If $j \in E\setminus A$, then $\partial_j\partial_A \mathcal M =\partial_{A\cup\{j\}}\mathcal M$.
  \item\label{lemma: partial dual properties: 4}  If $A'$ is some other subset of $E$, then $\partial_{A'}\partial_A \mathcal M=\partial_{A\symmdiff A'} \mathcal M$.
  \item  Partial duality preserves the orientability, and the connected components of maps.
  \item\label{lemma: partial dual properties: 5}  If $S$ is the underlying surface of $\partial_A \mathcal M$, then $S$ is also the underlying surface of $\partial_{E\setminus A} \mathcal M$.
\end{enumerate}
\end{lemma}

\subsection{Lagrangian matroids of partial duals}

It is not difficult to see that if $\mathcal B$ is the collection of bases of a map $\map$, then $\mathcal B^*$ will be the collection of bases of its dual map $\map^*$. We can also write $\mathcal B^*$ as the symmetric difference
\[\mathcal B^*=\mathcal B\symmdiff (E\cup E^*)=\{B^*\mid B\in\mathcal B\}.\]
Clearly, if $A$ is an admissible  $n$-subset of $E\cup E^*$, then the collection
\[\mathcal B\symmdiff (A\cup A^*)=\{B\symmdiff (A\cup A^*)\mid B\in\mathcal B\}\]
forms the collection of bases of a Lagrangian matroid. This Lagrangian matroid has a clear-cut geometrical interpretation: it is the Lagrangian matroid of the partial dual $\partial_{E\cap A}\map$.

\begin{lemma}\label{lemma: partial}Let $\mathcal M$ be a map, $E$ its set of edges and $A$ some subset of $E$. Then
\[\Delta (\partial_A \mathcal M)=\Delta (\mathcal M) \symmdiff (A\cup A^*).\]
\end{lemma}
\begin{proof}It is sufficient to show the lemma for $A=\{j\}$ since the general result will then follow from lemma \ref{lemma: partial dual properties} \eqref{lemma: partial dual properties: 3}.

If $j$ is in no base, then it is a contractible loop in $\mathcal M$, and in $\partial_j \mathcal M$ it is an edge incident to a degree 1 vertex. In that case, the lemma follows easily.

So let $B$ be a base of $\mathcal M$ with $j\in B$. Depending on if $j$ is a contractible segment, or a non-contractible loop in $\mathcal M$, in $\partial_j \mathcal M$ it will be a non-contractible loop, or a contractible segment, respectively.

Suppose first that $j$ is a contractible segment so that both vertices incident to it have degree at least 2. By our construction, $j$ is a loop in $\partial_j \map$. Therefore, the map $\map/j$ obtained from $\map$ by contracting $j$ is the same map as $(\partial_j \map)\setminus j$. The underlying surface of $\map/j$ is the surface of $\map$, hence $B\setminus j$ does not disconnect it. Therefore, $B\setminus j$ is a base of $(\partial_j \map)\setminus j$ as well.

Let us now adjoin the loop $j$ back to $(\partial_j \map)\setminus j$. If we were forced to add a handle, then $j^*\in \partial_j \map$ will not disconnect the underlying surface since it will split the new handle into two sleeves and leave the rest of the surface unaffected. Therefore, $(B\setminus j)\cup j^* = B\symmdiff \{j,j^*\}$ will be a base of $\partial_j \map$.

If a new handle was not needed, then $\partial_j \map$ and $(\partial_j \map)\setminus j$ are on the same surface. Adjoining $j$ to $(\partial_j \map)\setminus j$ will clearly split some face into two new faces, hence $j^*\in\partial_j \map$ must be a contractible segment since its endpoints are in the two faces with $j$ as a common boundary. Therefore, $B\symmdiff \{j,j^*\}$ is a base of $\partial_j \map$.\vspace{\topsep}

Now suppose that $j$ is a loop. Since $j\in B$, it cannot be contractible. If $\map$ and $\partial_j \map$ are on the same surface, then, topologically, $j\in \map$ and $j^*\in\partial_j \map$ are the same loop. Therefore, $B\symmdiff \{j,j^*\}$ must be a base of $\partial_j \map$.
Otherwise, by removing a handle, Euler's formula implies that $\partial_j \map$ gained an additional face. By construction, $j\in\partial_j \map$ must be on the boundary of the additional face, and at least one other face since other edges in $\map$ do not contribute to the partial dual. Therefore, $j^*\in\partial_j \map$ is a contractible segment, and hence $B\symmdiff \{j,j^*\}$ is a base for $\partial_j \map$.

So far we have shown that $\dtroid\map\symmdiff\{j,j^*\} \subseteq \dtroid{\partial_j \map}$. The other inclusion is obtained by noting that if $B\in\Delta(\partial_j \map)$, then
\[B\symmdiff \{j,j^*\} \in \Delta(\partial_j \map)\symmdiff \{j,j^*\}.\]
However, by using the just proven inclusion we have
\[B\symmdiff\{j,j^*\} \in \Delta (\partial_j\partial_j \map)=\Delta (\map).\]
Moreover, since $B=(B\symmdiff \{j,j^*\})\symmdiff \{j,j^*\}$, we must have $B\in \Delta(\map)\symmdiff\{j,j^*\}$.
\end{proof}


\begin{theorem}\label{thm:correspondence}The set of bases of $\map$, and hence the set of bases of $\dtroid\map$ is in a 1-1 correspondence with the partial duals of $\map$ which have exactly one face. The correspondence is given by $B\mapsto\partial_{E\setminus B}\map$.\end{theorem}
\begin{rmrk}Note that this correspondence is \emph{not} considered up to isomorphism, i.e.\ if $A_1\neq A_2$ we consider $\partial_{A_1}\map$ and $\partial_{A_2}\map$ to be distinct even if they happen to be isomorphic as maps.
\end{rmrk}
\begin{proof}First let us show that for a base $B\subseteq E$ of $\map$, the partial dual $\partial_{E\setminus B}\map$ has precisely one face.

If $B=E$ then $\map$ has exactly one face and $\partial_{E\setminus B}\map=\partial_\emptyset\map=\map$.
Let us assume that $B=\{1,2,\dots,n-1,n^*\}$. Then $B$ corresponds to a base $\{1,2,\dots,n-1, n^*\}$ of $\dtroid\map$, also denoted by $B$. Furthermore, $B\symmdiff \{n,n^*\}$ is a base of $\dtroid{\partial_n\map}$ and hence $\{1,2,\dots,n\}$ is a base of $\partial_n\map$. Therefore, $\partial_n\map=\partial_{E\setminus B}\map$ has precisely one face. The general case $B\subset E$ now follows by dualising $B$ with respect to $\{j,j^*\}$ for all $j\in E\setminus B$ and possibly relabelling.

Therefore, the map $B\mapsto\partial_{E\setminus B}\map$ is well defined. To see that it is a bijection, we only have to show that it is surjective. We have to show that if the partial dual $\partial_{E\setminus A}\map$ has exactly one face, then $A$ is a base of $\map$. Hence, if $\partial_{E\setminus A}\map$ has exactly one face, then $E$ is its base. That $A$ is a base of $\map$ now follows from lemma \ref{lemma: partial} with $E\setminus A$ in place of $A$ and the equality $E= A\symmdiff ((E\setminus A)\cup(E\setminus A)^*)$.  
\end{proof}

\section{Representations of partial duals and the action of $BC_n$}\label{section: representation of partial duals}

It should be of no surprise that the representations of the Lagrangian matroid of $\partial_j\map$ are closely related to the representations of the Lagrangian matroid of $\map$. In fact, we have the following

\begin{lemma}Let $\mathcal M$ be a map, and $(X|Y)$ a representation of $\dtroid{\mathcal M}$. Then a representation of $\dtroid {\partial_j \map}$ is obtained from $(X|Y)$ by permuting the columns $j$ and~$j^*$.\end{lemma}
\begin{proof}We may assume that $j=1$ by relabelling the edges of the map.
If $B$ is a base of $\dtroid\map$, then the corresponding base in $\partial_1 \mathcal M$ is $B\symmdiff\{1,1^*\}$, and hence all the linear independences between the column vectors corresponding to $B\setminus\{1,1^*\}$ have to be preserved in a representation of $\dtroid {\partial_1 \mathcal M}$. Therefore, we want to show that a representation of $\dtroid {\partial_1 \mathcal M}$ is given by $(P|Q)$ where $P$ and $Q$ coincide with $X$ and $Y$ on all but the first columns, respectively, and such that the first column of $P$ is the first column of $Y$, and the first column of $Q$ is the first column of $X$. To that effect, for $R\in\{X, Y, P, Q\}$ let $R'$ denote the matrix $R$ with its first column removed.

By attaching the column $1^*$ to $P'$ and the column $1$ to $Q'$ as the first columns of $P$ and $Q$, respectively, the admissible column vectors $V=\{1,v_2,\dots,v_n\}$ in $(X|Y)$ are the admissible column vectors $W=\{1^*,v_2,\dots,v_n\}$ in $(P|Q)$. Therefore, if $V$ corresponds to a base $B_V$ of $\dtroid\map$, then $W$ corresponds to a base $B_W$ of $\dtroid{\partial_1 \mathcal M}$ with $B_W=B_V\symmdiff\{1,1^*\}$. The same argument holds if we replace $1$ with $1^*$, and $1^*$ with $1$ in $V$ and $W$, respectively.\end{proof}

\begin{theorem}\label{theorem:representations}Let $\mathcal M$ be a map, $E$ its set of edges, and $A\subseteq E$. If $(X|Y)$ is a representation of $\dtroid\map$, then a representation of $\dtroid {\partial_A\map}$ can be obtained from $(X|Y)$ by permuting the columns $j$ and $j^*$, for all $j\in A$.\end{theorem}

\subsection{The action of the Coxeter group $BC_n$}\label{section:action}

There is an equivalent definition of Lagrangian matroids in the language of Coxeter groups; let $BC_n$ be the group of all permutations $w$ of the set $[n]\cup[n]^*$ such that $w(i^*)=w(i)^*$. This group is isomorphic to the group of symmetries of the $n$-cube $[-1,1]^n$, so that $i\in[n]\cup[n]^*$ is one of its facets, and $i$ and $i^*$ are opposite facets.

When $n=3$ the dual polytope of $[-1,1]^3$ is an octahedron and we can identify $[3]\cup[3]^*$ with its set of vertices so that $i$ and $i^*$ are the centres of opposite faces of $[-1,1]^3$. In general, the dual polytope of the $n$-cube $[-1,1]^n$ is called an $n$-\emph{hyperoctahedron}. Its symmetry group is $BC_n$, and for that reason $BC_n$ is sometimes called the \emph{hyperoctahedral group}.

It is well known that $BC_n$ is generated by the transpositions $(j~j^*)$ and the products $(j~k)(j^*~k^*)$, for all $j,k\in[n]\cup[n]^*$ with $j\notin \{k,k^*\}$.

\begin{rmrk}\label{remark:gen}Note that if $j\in[n]$ and $k\in[n]^*$, then \[(j~k)(j^*~k^*)=(j~k^*)(j^*~k)(j~j^*)(k~k^*).\] This equality will be useful when we will be considering the action of $BC_n$ on maps.
\end{rmrk}

The generators of $BC_n$ act on a Lagrangian matroid $\mathcal B$ by acting on its bases in the following way: for $A\in\mathcal B$ 
\begin{itemize}
	\item a generator $(j~j^*)$ acts on $A$ by interchanging $j \leftrightarrow j^*$,
	\item a generator $(j~k)(j^*~k^*)$ acts on $A$ by interchanging $j \leftrightarrow k$ and $j^*\leftrightarrow k^*$.
\end{itemize} 

We can reinterpret the symmetric exchange axiom as the action of $BC_n$ on the elements of $\mathcal B$ by transpositions $(j~j^*)$.

\begin{definition}\emph{Symmetric exchange axiom.}
For any $A,B\in\mathcal B$, and $j\in A\symmdiff B$, there exists $k\in B\symmdiff A$ such that $(j~j^*)(k~k^*)A\in\mathcal B$. In the case when $j=k$ we consider only $(j~j^*)A$.\end{definition}

Similarly, we can reinterpret lemma \ref{lemma: partial} as $BC_n$ acting on the bases by transpositions $(j~j^*)$, i.e.
\[\dtroid{\partial_j \map}=(j~j^*)\dtroid\map=\{(j~j^*)B\mid B\in\dtroid\map\},\]
and if $A=\{j_1,\dots,j_k\}\subseteq E$, then we have
\[\dtroid{\partial_A \map}=(j_1~j_1^*)\cdots(j_k~j_k^*)\dtroid\map=\{(j_1~j_1^*)\cdots(j_k~j_k^*)B\mid B\in\dtroid\map\}.\]

In fact, we can lift this action of $BC_n$ to an action on maps with $n$ edges. The action of a generator $(j~j^*)$ on $\map$ results with $\partial_j\map$ while the action of a generator $(j~k)(j^*~k^*)$, when $j$ and $k$ are both in $[n]$, amounts to relabelling the edges $j$ and $k$, together with their respective coedges. In terms of the cartography group $\pergrp$ of $\map$, if the edge $j$ corresponds to the transposition $c_j=(2j-1 ~ 2j)$ so that $\alpha=(1~2)(3~4)\cdots (2j-1 ~ 2j)\cdots (2n-1 ~ 2n)$ we have:
\[(j~j^*)\pergrp = \pergrpj,\]
\[(j~k)(j^*~k^*)\pergrp = \langle \sigma^{d_{jk}},\alpha^{d_{jk}},\varphi^{d_{jk}}\rangle,\]
where $g^{d_{jk}}$ denotes conjugation by $d_{jk}=(2j-1 ~ 2k-1)(2j~2k)$ for $g\in\{\sigma,\alpha,\varphi\}$.

If $j\in[n]$ and $k\in[n]^*$, then by remark \ref{remark:gen} we can write $(j~k)(j^*~k^*)$ as $(j~k^*)(j^*~k)(j~j^*)(k~k^*)$, with $k^*\in[n]$. Therefore, the action of $(j~k)(j^*~k^*)$ amounts to dualising $\map$ with respect to the edges $j$ and $k^*$ (since $k\in[n]^*$, $k^*$ is an edge of $\map$, not of $\map^*$), and subsequently interchanging the labels for $j$ and $k^*$. In terms of the cartography group of $\map$ we have
\[(j~k)(j^*~k^*)\pergrp = \langle (\sigma c_jc_{k^*})^{d_{jk^*}},\alpha^{d_{jk^*}},(c_j c_{k^*}\varphi)^{d_{jk^*}}\rangle.\]

The two actions clearly are compatible since for $w\in BC_n$ we have
\[\dtroid{w\map}=w\dtroid\map.\]

\subsection{The $C_n$ root system, Lagrangian matroid polytopes and partial duals}\label{section:polytopes}

Lagrangian matroids are closely related to the root system of type $C_n$, i.e.\ the root system of the hyperoctahedral group $BC_n$, which is defined as follows.

Let $\{\epsilon_1,\dots,\epsilon_n\}$ denote the standard orthonormal basis for $\mathbb R^n$ and set $\epsilon_{i^*}=-\epsilon_i$ for all $i^*\in[n]^*$. The \emph{root system of type} $C_n$ is given by the vectors $\pm2\epsilon_i$ and $\pm\epsilon_i\pm\epsilon_j$ for all $1\leq i<j\leq n$.

For an admissible $n$-set $A=\{i_1,\dots,i_n\}\subset[n]\cup[n]^*$ let $v_A\in\mathbb R^n$ be the point
\[v_A=\epsilon_{i_1}+\cdots+\epsilon_{i_n}.\]
If $\mathcal B$ is a collection of admissible $n$-sets in $[n]\cup[n]^*$, let $\mathbf P$ denote the convex hull of the points $v_A$, for all $A\in\mathcal B$.

\begin{theorem}[Gelfand-Serganova]The collection $\mathcal B$ is the collection of bases of a Lagrangian matroid if, and only if all edges of $\mathbf P$ are parallel to the roots in the root system of type $C_n$.
\end{theorem}

We omit the proof, however an interested reader may look it up in \cite[Chapter 3.3]{borovik_gelfand_white}.

Given a Lagrangian matroid $\mathcal B$, we say that $\mathbf P$, or $\mathbf P(\mathcal B)$ is its \emph{Lagrangian matroid polytope}. The hyperoctahedral group acts on $\mathbf P$ in the following way:
\begin{itemize}
	\item a generator $(j~j^*)$ acts on $\mathbf P$ by reflection with respect to the hyperplane orthogonal to the root $\pm2\epsilon_j$,
	\item a generator $(j~k)(j^*~k^*)$ acts on $\mathbf P$ by reflection with respect to the hyperplane orthogonal to the root $\epsilon_j-\epsilon_k$.
\end{itemize}

This action is compatible with the action of $BC_n$ on Lagrangian matroids (and hence with its action on maps) in the following sense: for $w\in BC_n$ if $\mathcal B$ is a Lagrangian matroid and $\mathbf P(\mathcal B)$ its Lagrangian matroid polytope, then \[w\mathbf P(\mathcal B)=\mathbf P(w\mathcal B).\]

In the following example we consider a map $\map$, compute its Lagrangian matroid and draw its Lagrangian matroid polytope. Then we consider the action of $BC_n$ by transpositions on all three objects.

\begin{example}Let $\map$ be a genus 0 map with 3 vertices, 3 edges and 2 faces as shown and labelled in figure \ref{fig:map}.

\begin{figure}[ht]
	\centering
		\includegraphics[trim=5cm 18.25cm 5cm 5cm]{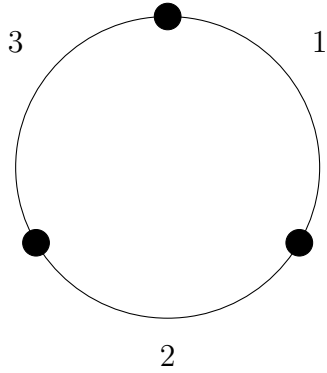}
	\caption{A genus 0 map with 3 vertices, 3 edges and 2 faces.}
	\label{fig:map}
\end{figure}

Since $\map$ is genus 0, then its Lagrangian matroid $\dtroid\map$ is a matroid, hence its bases are in a 1-1 correspondence with its spanning trees. Therefore,
\[\dtroid\map=\{123^*, 12^*3, 1^*23\}.\]
A representation of $\dtroid\map$ is given by the following matrix:
\[
\bordermatrix{
	 & 1 & 2 &  3 & 1^* & 2^* & 3^* \cr
   & 1 & 0 & -1 & 0 & 0 & 0\cr
   & 0 & 1 & -1 & 0 & 0 & 0\cr
   & 0 & 0 &  0 & 1 & 1 & 1}.
	\]

The Lagrangian matroid polytope $\mathbf P(\dtroid\map)$ is the convex hull of the points $(1,1,-1)$, $(1,-1,1)$ and $(-1,1,1)$ given in figure \ref{fig:poly}.

\begin{figure}[ht]
	\centering
		\includegraphics[trim=5cm 16.5cm 5cm 5cm]{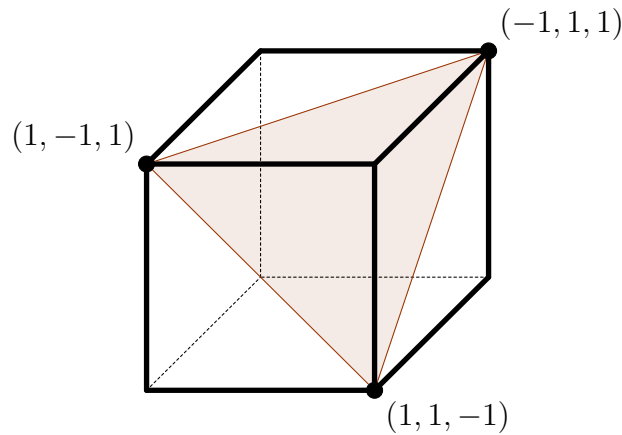}
	\caption{The Lagrangian matroid polytope $\mathbf P(\dtroid\map)$. The point $(\pm 1, \pm 1, \pm 1)$ is the intersection of the front/back, right/left and top/bottom faces of the cube $[-1,1]^3$, respectively.}
	\label{fig:poly}
\end{figure}

Let us consider the partial duals $\partial_1\map$, $\partial_2\map$ and $\partial_3\map$. Note that per lemma \ref{lemma: partial}\eqref{lemma: partial dual properties: 4} all other partial duals are the full duals of the maps $\map$ and $\partial_j\map$ for $j\in\{1,2,3\}$, hence we will not be drawing them.

The three partial duals $\partial_j\map$ for $j\in\{1,2,3\}$ are all genus 1 maps, as shown in figure \ref{fig:pduals}. In fact, they are all isomorphic as maps. Their Lagrangian matroids are
\begin{align*}
\dtroid{\partial_1\map}&=(1~1^*)\dtroid\map=\{1^*23^*, 1^*2^*3, 123\},\\
\dtroid{\partial_2\map}&=(2~2^*)\dtroid\map=\{12^*3^*, 123, 1^*2^*3\},\\
\dtroid{\partial_3\map}&=(3~3^*)\dtroid\map=\{123, 12^*3^*, 1^*23^*\}.
\end{align*}
The isomorphisms $(f_{ij},f_{ij}^*)$ between the Lagrangian matroids $\dtroid{\partial_i\map}$ and $\dtroid{\partial_j\map}$, written as permutations of $[n]$ and $[n]^*$, are given by $f_{ij}=(i~j)$ and $f^*_{ij}=(i^*~j^*)$, for $1\leq i<j\leq3$.

\begin{figure}[ht]
	\centering
		\includegraphics[scale=.75,trim=8cm 18cm 8cm 4cm]{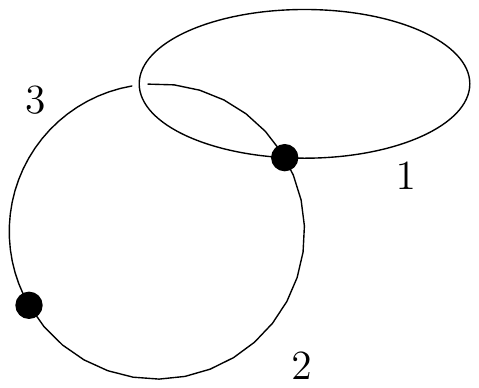}\includegraphics[scale=.75,trim=8cm 18cm 8cm 4cm]{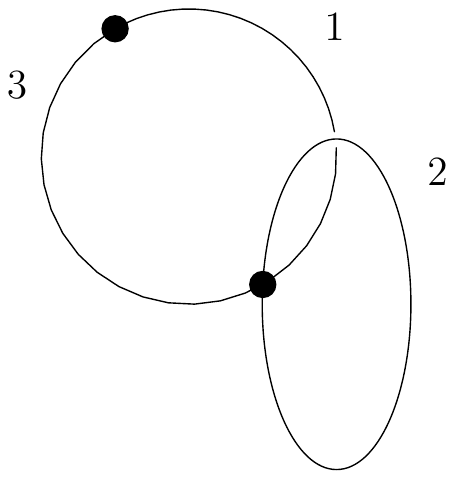}\includegraphics[scale=.75,trim=8cm 18cm 8cm 4cm]{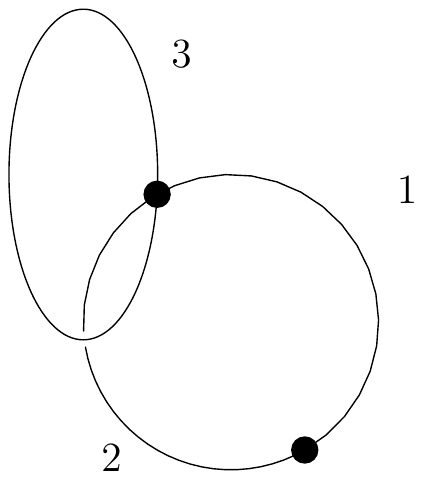}
	\caption{From left to right, the partial duals $\partial_1\map$, $\partial_2\map$ and $\partial_3\map$. All three maps are of genus 1.}
	\label{fig:pduals}
\end{figure}

The polytopes $\mathbf P(\dtroid{\partial_j\map})$ shown in figure \ref{fig:polypartial} are obtained from $\mathbf P(\dtroid\map)$ by reflecting it with respect to the hyperplanes orthogonal to the roots $\pm2\epsilon_j$, for $j\in\{1,2,3\}$.

\begin{figure}[ht]
	\centering
		\includegraphics[scale=.5,trim=6cm 14.5cm 8cm 6cm]{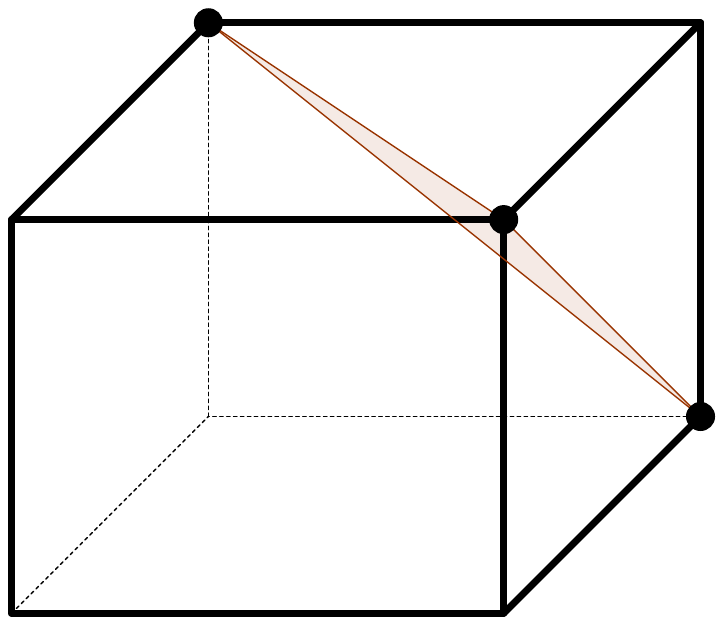}\includegraphics[scale=.5,trim=6cm 14.5cm 8cm 6cm]{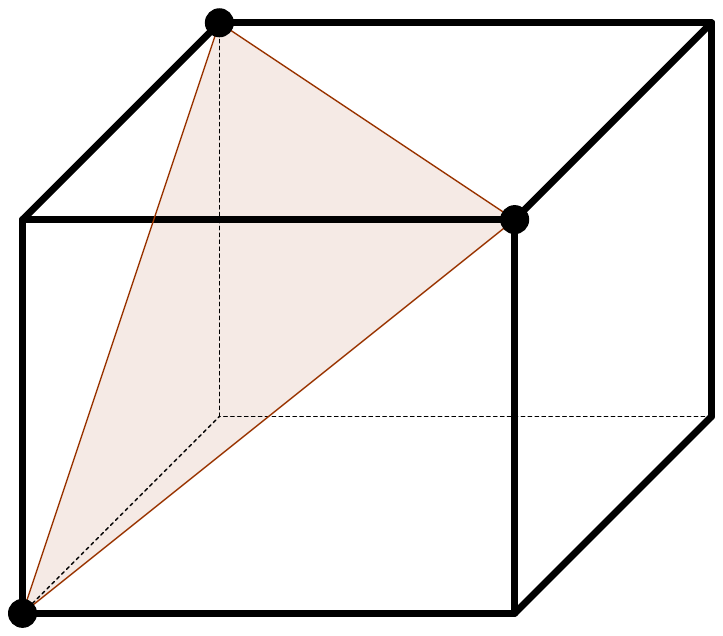}\includegraphics[scale=.5,trim=6cm 14.5cm 8cm 6cm]{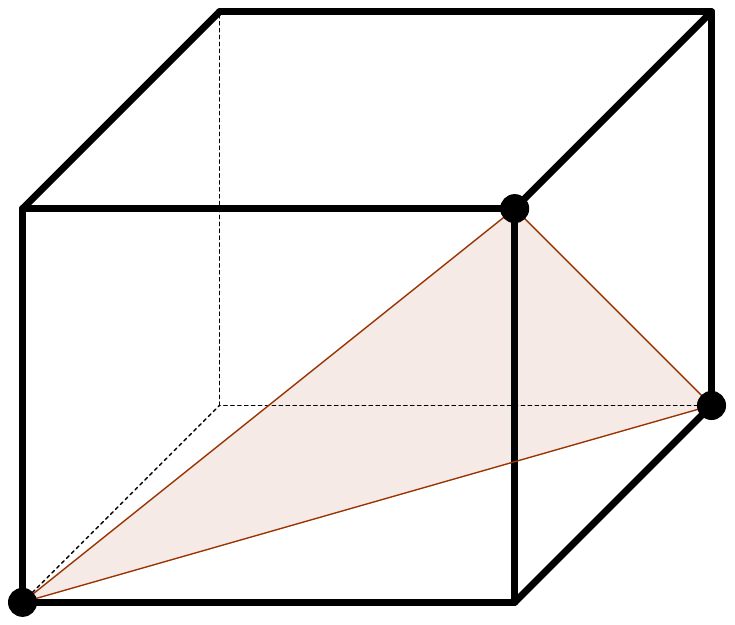}
	\caption{From left to right, the Lagrangian matroid polytopes $\mathbf P(\dtroid{\partial_1\map})$, $\mathbf P(\dtroid{\partial_2\map})$ and $\mathbf P(\dtroid{\partial_3\map})$.}
	\label{fig:polypartial}
\end{figure}

Note that the polytopes $\mathbf P(\dtroid{\partial_j\map})$ form three faces of a tetrahedron inscribed into $[-1,1]^3$. The missing face corresponds to $\mathbf P(\dtroid{\map^*})$. By duality, the four polytopes $\mathbf P(\dtroid{\map})$, $\mathbf P(\dtroid{\partial_{12}\map})$, $\mathbf P(\dtroid{\partial_{13}\map})$ and $\mathbf P(\dtroid{\partial_{23}\map})$ form the other inscribed tetrahedron, hence the set of Lagrangian matroid polytopes of all partial duals of $\map$ forms the stellated octahedron, as shown in figure \ref{fig:stellar}.

\begin{figure}[ht]
	\centering
		\includegraphics[scale=.65,trim={1.5cm 1.5cm 1cm .5cm}]{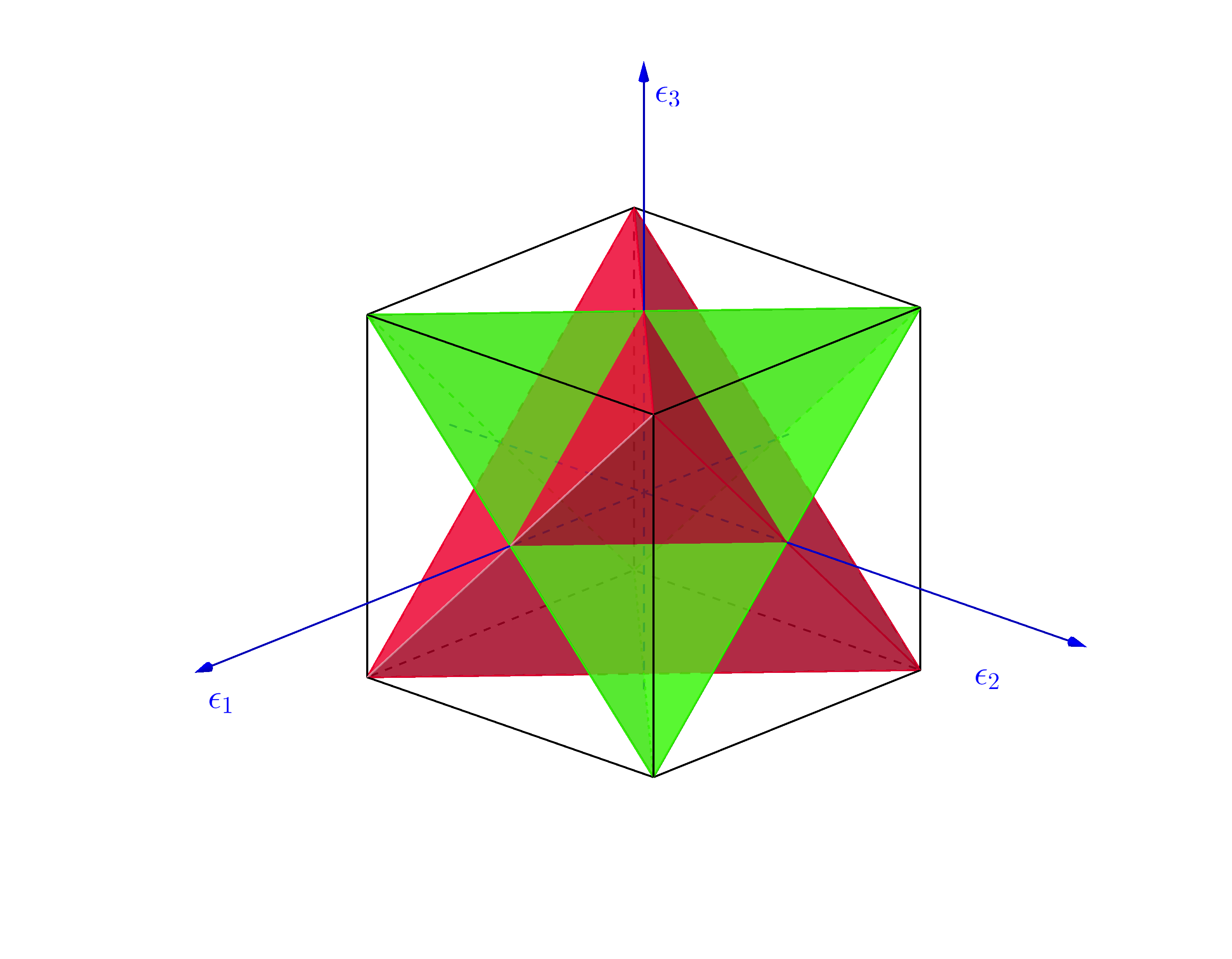}\includegraphics[scale=.65,trim={.75cm 1.75cm 2.5cm .25cm}]{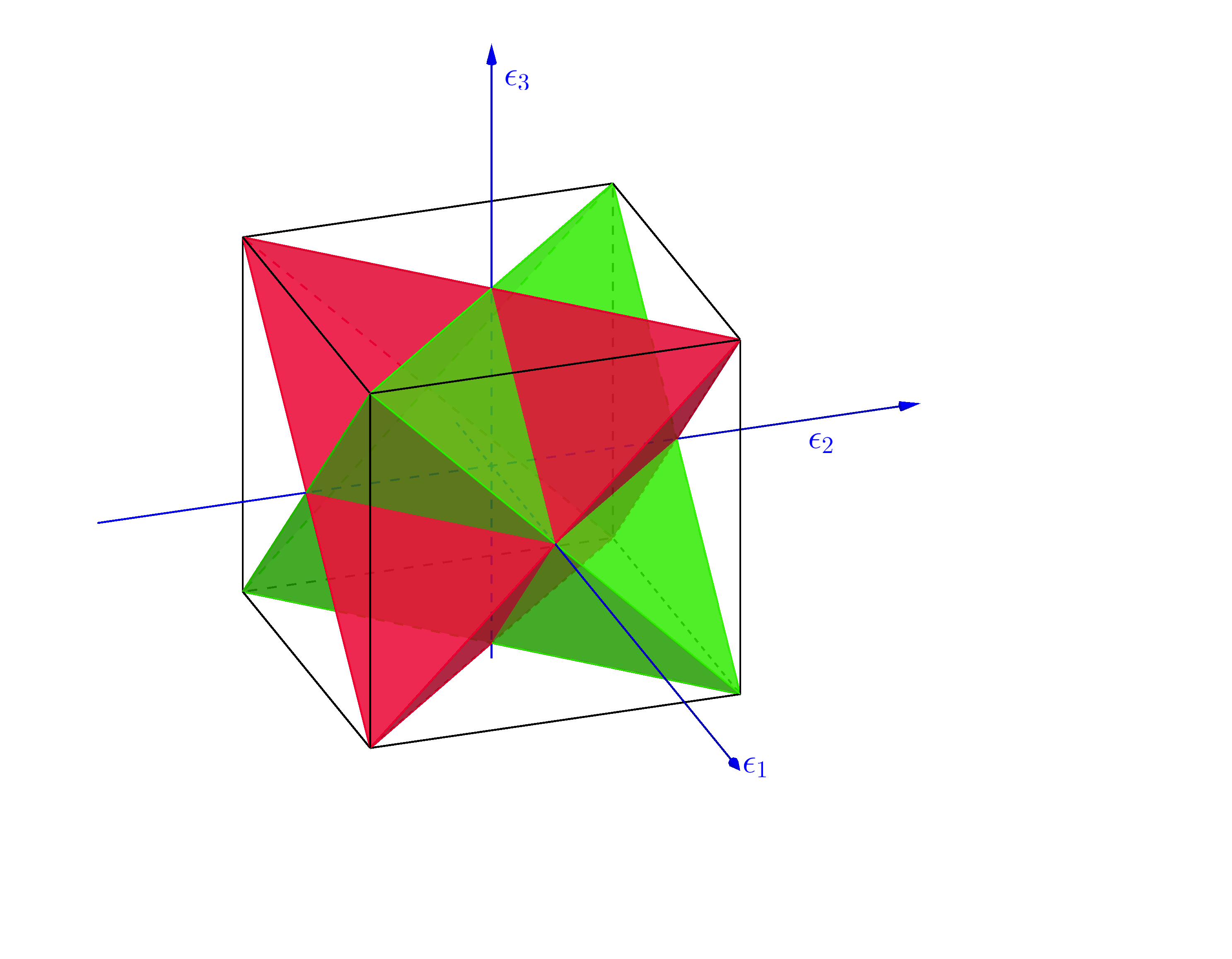}
	\caption{A stellated octahedron formed by the Lagrangian matroid polytopes of the partial duals of $\map$. The green faces (the green tetrahedron) correspond to the polytopes $\mathbf P(\dtroid{\map})$, $\mathbf P(\dtroid{\partial_{12}\map})$, $\mathbf P(\dtroid{\partial_{13}\map})$ and $\mathbf P(\dtroid{\partial_{23}\map})$, while the red faces (the red tetrahedron) correspond to the polytopes $\mathbf P(\dtroid{\map^*})$, $\mathbf P(\dtroid{\partial_{1}\map})$, $\mathbf P(\dtroid{\partial_{2}\map})$ and $\mathbf P(\dtroid{\partial_{3}\map})$.}
	\label{fig:stellar}
\end{figure}
\end{example}

\subsection{Lagrangian pairs}\label{section:pairs}

In this section we show that Lagrangian subspaces representing maps on surfaces and their partial duals are a very natural example of a pair of Lagrangian subspaces, which we define below.

Let $V$ be an orthogonal $2n$-space. Two Lagrangian subspaces of $V$ are said to be a \emph{pair of Lagrangian subspaces} if their intersection has dimension $n-1$. It is well known that if $W<V$ is a totally isotropic subspace of dimension $n-1$, then it is contained in exactly two Lagrangian subspaces.

The following theorem is established in \cite[pp.~1031]{Booth20051023}:
\begin{theorem}\label{theorem:pair}Let $\mathcal B_1$ be the collection of bases of a Lagrangian orthogonal matroid, and $\mathcal B_2 = (i~i^*)\mathcal B_1$, where $(i~i^*)$ is to be thought of as a permutation acting element-wise on the members of $\mathcal B_1$. Then $\mathcal B_1$ and $\mathcal B_2$ are a Lagrangian pair of orthogonal matroids.\end{theorem}

Recall from the proof of theorem \ref{theorem:maps} that a map $\map$ on a surface $S$ with its set $V\cup V^*$ of vertices and covertices gives rise to a Lagrangian subspace $\iota(H_1(S\setminus(V\cup V^*)))$ of $\q^E\oplus\q^{E^*}$, where $\iota\colon H_1(S\setminus(V\cup V^*))\to\q^E\oplus\q^{E^*}$ assigns to a cycle $c$ its incidence vector in $\q^E\oplus\q^{E^*}$.

The following is clear from theorem \ref{theorem:pair}.
\begin{theorem}Let $S_j$, $V_j$ and $V_j^*$ denote the underlying surface of $\partial_j \mathcal M$ and its sets of vertices and covertices, respectively. Let
\[\mathcal H_j=H_1(S_j\setminus\{V_j\cup V_j^*\}),\]
for $j=0,1,\dots,n$, where $\partial_0 \mathcal M=\mathcal M$, $S_0=S$, $V_0=V$ and $V_0^*=V^*$. The pair $(\iota(\mathcal H_0),\iota(\mathcal H_j))$ is a Lagrangian pair of subspaces of $\orth$ whenever $j\neq 0$.\end{theorem}

Note that this result does not hold for an arbitrary subset $A$ of $E$ or $E\cup E^*$ since the intersection of their respective Lagrangian subspaces may not be $n-1$ dimensional. In some cases $\map$ and $\partial_A\map$ can be isomorphic (see example \ref{example:dessin}) so the intersection of their respective Lagrangian subspaces is $n$ dimensional.

\section{Grothendieck's dessins d'enfants}\label{section: application}

A \emph{dessin d'enfant}, or just \emph{dessin} for short, is a pair $(X,f)$, where $X$ is an algebraic curve, or equivalently a compact Riemann surface, and $f\colon X\to\rsphere$ is a holomorphic ramified covering, ramified at most over a subset of $\zoi$. Moreover, by \belyi's theorem \cite{belyi80, belyi02, girondo_gonzalez-diez, lando_zvonkin, schneps94}, both $X$ and $f$ have a model over $\algbr$. Two dessins are isomorphic if they are isomorphic as ramified coverings.

It is well known that isomorphism classes of dessins are in a 1-1 correspondence with isomorphism classes of bipartite maps on surfaces, with vertices coloured in black and white. The black and white vertices correspond to the points in $\inv f(0)$ and $\inv f(1)$, respectively, the half-edges correspond to the preimages of the open unit interval, and the points in $\inv f(\infty)$ correspond to face-centres. If all white vertices of a dessin have degree 2, then the dessin is a map since two half-edges glue along a white vertex into an edge.

\begin{rmrk}Note that a segment on $X$ connecting a black with a white vertex is regarded not as an edge, but as a half-edge. When a white vertex is incident to precisely two half-edges, then we consider the two half-edges to be a single edge. Otherwise, if more than two half-edges are adjacent to a white vertex, then we no longer have an edge, but rather a \emph{hyper-edge}.
\end{rmrk}

In general, the equations defining a dessin $(X,f)$ can be written down in many ways. The field extension of $\q$ generated by the coefficients of $X$ and $f$ is called a \emph{field of definition for} $(X,f)$. \belyi's theorem guarantees that at least one field of definition is a subextension of $\algbr$.

There is a natural action of the absolute Galois group $\absgal$ over $\q$ on a dessin $(X,f)$: an automorphism $\theta\in\absgal$ acts on $(X,f)$ by acting on the coefficients of both $X$ and $f$. The image $(X,f)^\theta=(X^\theta,f^\theta)$ under the action is another dessin, and moreover, if $(X,f)$ is a map, then $(X,f)^\theta$ is a map as well.

\begin{rmrk}Note that by considering maps only we do not lose any Galois-theoretic information encoded by dessins: if $(X,f)$ is not a map, i.e.\ if it has a white vertex of degree not equal to 2, then the dessin $(X,\beta)$, where $\beta=4f(1-f)$ is defined over the same extension of $\q$ as $(X,f)$ and moreover, it is a map. Moreover, $\beta$ sends the points of $\inv f(0)$ and $\inv f(1)$ to 0, and $\inv\beta(1)=\inv f(1/2)$, hence the map corresponding to $(X,\beta)$ is obtained from $(X,f)$ by colouring all the white vertices black and adding new white vertices of degree 2 to each edge.
\end{rmrk}

An interesting class of partial duals of $\map$ are the partial duals
$\partial_{E\setminus B}\map$, where $B$ is a base of $\dtroid\map$. By theorem \ref{thm:correspondence} the partial dual $\map_B$ has precisely one face. Such maps are of great interest in the theory of dessins d'enfants as they provide examples of maps which can be defined over their \emph{field of moduli}, which is defined as follows.

Let $(X,f)$ be a dessin, and $\txn{Stab}(X,f)$ its stabiliser in $\absgal$, i.e.
\[\txn{Stab}(X,f)=\{\theta\in\absgal\mid(X,f)\cong(X,f)^\theta\}.\]
The subfield of $\algbr$ fixed by $\txn{Stab}(X,f)$, i.e.
\[\txn{Fix}(X,f)=\{a\in\algbr\mid\theta(a)=a,\txn{ for all }\theta\in\txn{Stab}(X,f)\}\]
is called \emph{the field of moduli of} $(X,f)$.

The field of moduli is \emph{the best} field of definition for a dessin in the sense that it is contained in every field in which a model for $(X,f)$ can be written down.

However, it is not always possible to write down a model for $(X,f)$ over its field of moduli, so it is of importance to determine the necessary and sufficient criteria which a dessin has to satisfy so that its field of moduli is also a field of definition. One such was given by Birch \cite{birch1994noncongruence}: a dessin $(X,f)$ can be defined over its field of moduli if there is a point in $\inv f(\infty)$ whose ramification index is unique among points in $\inv f(\infty)$ (see also \cite{SijslingVoight}). This means that a dessin can be defined over its field of moduli if it has a face of unique degree, or equivalently, if the permutation $\varphi$ has a cycle of unique length. Therefore, we have the following
\begin{theorem}Let $\map$ be a map with $\mathcal B$ as its collection of bases. Then for any $B\in\mathcal B$ the partial duals
\[\partial_{E\setminus B}\map\txn{ and }\partial_{B\cap E}\map\]
can both be defined over their fields of moduli which coincide.
\end{theorem}
\begin{proof}We have established in theorem \ref{thm:correspondence} that the partial dual $\partial_{E\setminus B}\map$ has precisely one face, so it can be defined over its field of moduli. That $\partial_{B\cap E}\map$ can be defined over its field of moduli follows from the fact that if a map can be defined over its field of moduli, then so can its dual map and the two fields coincide: if $\map$ is given by the pair $(X,f)$, then its dual map will be given by the pair $(X,1/f)$. Since $\partial_{B\cap E}\map$ is the dual of $\partial_{E\setminus B}\map$, the theorem follows.\end{proof}

Unfortunately, since this theorem holds for all dessins that are maps, including those that cannot be defined over their field of moduli, it heavily suggests that matroids and partial duals will not be helpful in establishing criteria for $(X,f)$ to be realisable over its field of moduli. Nevertheless, some interesting phenomena may be observed, as shown in the following section.

\subsection{The action of $BC_n$ on dessins of degree $2n$}

We say that a dessin $(X,f)$ is of degree $n$ if the size of the fibre $\inv f(1/2)$ above $1/2$ is $n$. Equivalently, $(X,f)$ is of degree $n$ if its realisation as a bipartite map has precisely $n$ half-edges. If $(X,f)$ is a map, then it is of even degree $2n$.

In section \ref{section:action} we have introduced an action of $BC_n$ on maps:
\begin{itemize}
	\item a generator $(j~j^*)$ acts on $\map$ by partially dualising it with respect to the edge $j$,
	\item if $j,k\in[n]$, then $(j~k)(j^*~k^*)$ acts by relabelling,
	\item and if $j\in[n]$ but $k\in[n]^*$, then it acts by partially dualising with respect to $j$ and $k^*$ and relabelling.
\end{itemize}
Relabelling of a map has no effect on the dessin it induces, so the interesting action is carried by the transpositions $(j~j^*)$.

As opposed to $\absgal$, the action of $BC_n$ is topological in the sense that we can immediately tell from $\map$ how $(j~j^*)\map$ sits on a surface, whereas to see $\map^\theta$ for some $\theta\in\absgal$, we first have to find a model $(X,f)$ for it and then look at the preimage of $[0,1]$ under $\inv{(f^\theta)}$. On the other hand, the algebraic properties which both $\map$ and $(j~j^*)\map$ have in common are completely non-obvious, as their underlying algebraic curves can (and they often do) belong to distinct moduli spaces, as we shall see in the following example.

\begin{example}\label{example:dessin}Consider the Galois orbit shown in figure \ref{figure: appendix} consisting of three maps $\map^+$, $\map^-$ and $\map^\mathbb R$. These maps correspond to the dessins $(X_\nu,f_\nu)$ with
\begin{align*}
X_\nu\colon y^2 &=-\frac{17\nu^2 + 8 - 42\nu}{960400}
(19600x^2+55552\nu x -18432\nu^2x\\ &-88408x+338963-130592\nu+65792\nu^2)(x-1),\end{align*}
\begin{align*}
\frac{1}{f_\nu(x,y)} 	&=\frac{1}{735306250}(552\nu^2 - 617\nu + 68)(42875x^4 + 1756160\nu yx^2\\
									& - 860160\nu^2yx^2 - 4543840yx^2 + 3959200\nu x^3 - 10346175x^3\\
									& - 1926400\nu^2x^3 + 31782912\nu^2yx - 63438592\nu yx + 168996968yx\\
									& + 18916352\nu^2x^2- 37781632\nu x^2 + 100206428x^2 - 257512128y\\
									&- 48381952\nu^2y + 96684032\nu y - 62101504\nu^2x - 330259656x\\
									&+ 123960064\nu x + 48381952\nu^2 - 96684032\nu + 257512128),
\end{align*}
where $\nu$ is a root of $256\nu^3-544\nu^2+1427\nu-172$. The notation suggests that $\map^+$ and $\map^-$ correspond to the root with the positive and negative imaginary part, respectively, and that $\map^\mathbb{R}$ corresponds to the real root.

\begin{figure}[ht]
  \centering
  \includegraphics[scale=.4,trim={6cm 14cm 6cm 4.75cm}]{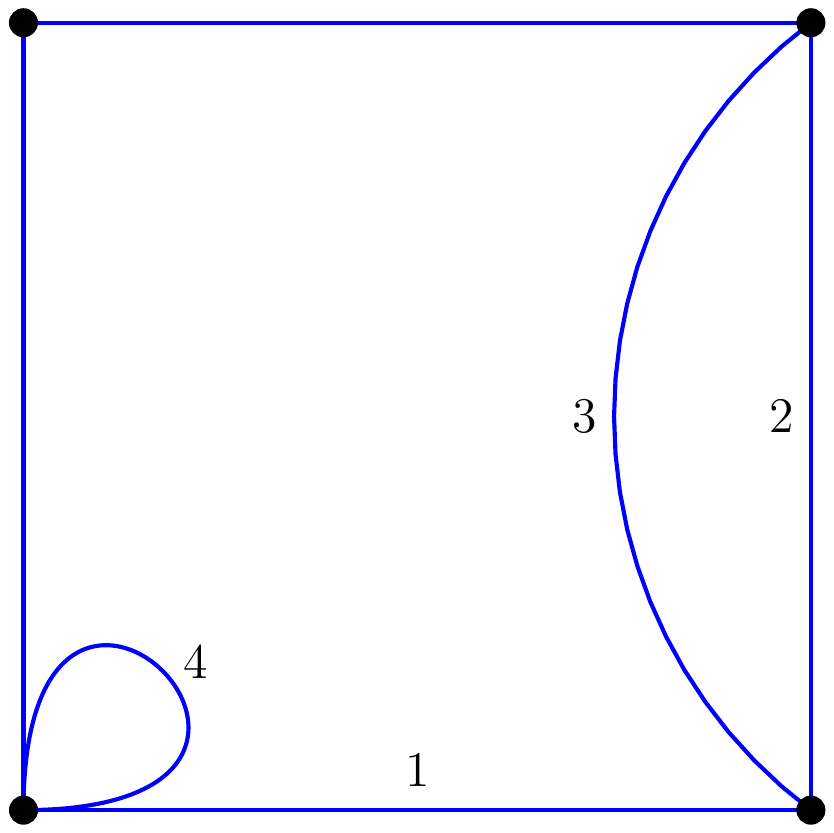}
  \includegraphics[scale=.4,trim={6cm 14cm 6cm 4.75cm}]{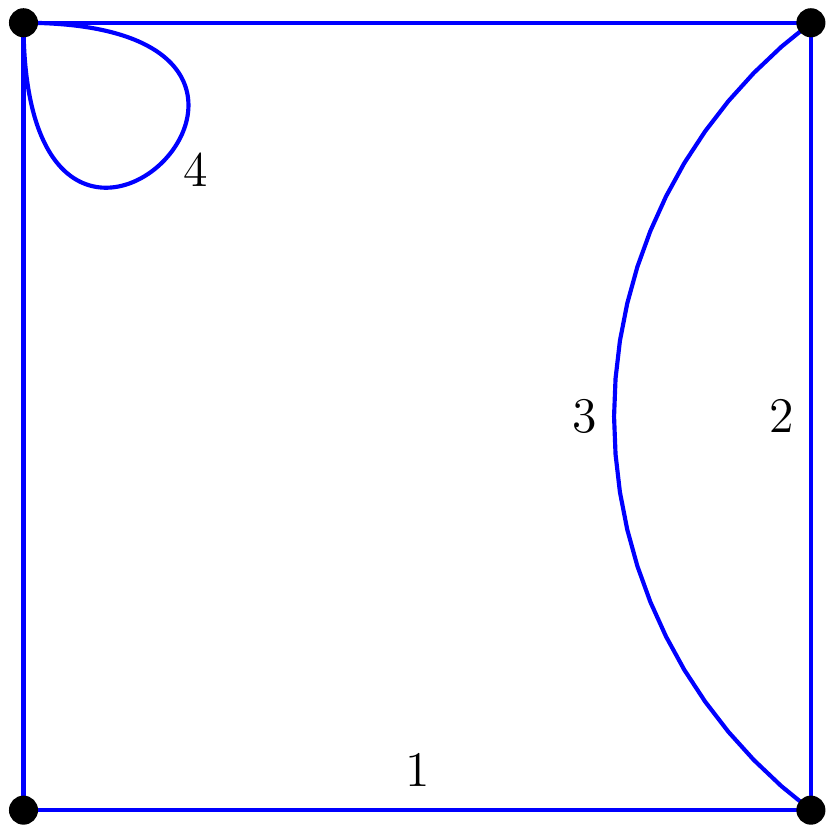}
  \includegraphics[scale=.4,trim={6cm 14cm 6cm 4.75cm}]{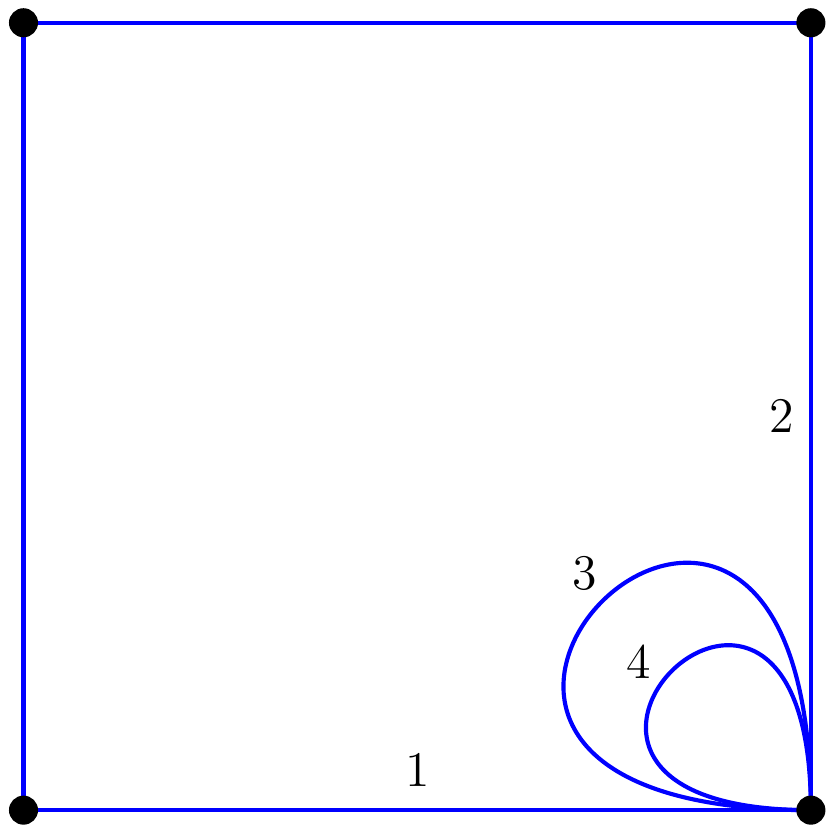}
  \caption{From left to right: the three genus 1 maps $\map^+$, $\map^-$ and $\map^\mathbb R$ with one vertex, four edges and two faces. These maps form a complete Galois orbit.}\label{figure: appendix}
\end{figure}

The maps $\map^+$ and $\map^-$ are each others partial duals with $\partial_{12}\map^-=\map^+$ so it suffices to look at only one of them. In what follows we shall illustrate the symmetry between the fields of definition and the cartographic groups of the partial duals of the maps $\map^-$ and $\map^\mathbb{R}$.

The fields of definition were computed in \cite{adrianov09}, while the cartography groups were computed by using Sage and GAP. The GAP function IdGroup was used to identify the isomorphism types of the cartographic groups in question. All of them appear in GAP's \emph{SmallGroups} library and unless they belong to a familiar family, we refer to them as $(m,n(m))$, where $m$ is the order of the group and $n(m)$ is its unique identifier among all groups of order $m$.

To illustrate the symmetry of fields of definition and cartography groups of partial duals, we shall use the following schematic:
\begin{table}[ht]
  \centering
  \scalebox{.95}{
   \begin{tabular}{ccccccccccc}
     &  &  &  &  & $\map$ &  &  &  &  & \\[5pt]
     & $\partial_1\map$ &  & $\partial_2\map$ &  &  &  & $\partial_3\map$ &  & $\partial_4\map$ &  \\[5pt]
     $\partial_{12}\map$ &  & $\partial_{13}\map$ &  & $\partial_{14}\map$ &  & $\partial_{23}\map$ &  & $\partial_{24}\map$ &  & $\partial_{34}\map$ \\[5pt]
     & $\partial_{123}\map$ &  & $\partial_{124}\map$ &  &  &  & $\partial_{134}\map$ &  & $\partial_{234}\map$ &  \\[5pt]
     &  &  & &  & $\partial_{1234}\map$ &  &  &  &  &
   \end{tabular}
  }
  \caption{The schematic for the partial duals of a map $\map$.}\label{table: maps}
\end{table}

The symmetry of the fields of definition and the cartography groups is accounted for by duality: dual maps have coinciding fields of definition and isomorphic cartography groups, and the partial dual $\partial_A\map$ is dual to $\partial_{\{1,2,3,4\}\setminus A}\map$.

We also note that the pairs $\{\partial_3\map^-,\partial_{123}\map^-\}$ and $\{\partial_4\map^-,\partial_{124}\map^-\}$ form two complete Galois orbits of order 2 dual to each other. An automorphism acting on these maps is given by $i\sqrt 7 \mapsto -i\sqrt 7$.

In contrast, no two partial duals of $\map^\mathbb{R}$ belong to the same Galois orbit.\vspace{\topsep}

\noindent\emph{\textbf{Overview of fields of definition.}} The field of definition of $\map^-$ and $\map^\mathbb{R}$ is $\q(\nu^{\pm\mathbb R})$, where $\nu^{\pm\mathbb R}$ denotes the set of roots of $256\nu^3-544\nu^2+1427\nu-172$.

\begin{table}[ht]
  \centering
  \scalebox{1}{
   \begin{tabular}{ccccccccccc}
     &  &  &  &  & $\q(\nu^{\pm\mathbb R})$ &  &  &  &  & \\[5pt]
     & $\q$ &  & $\q$ &  &  &  & $\q(i\sqrt 7)$ &  & $\q(i\sqrt 7)$ &  \\[5pt]
     $\q(\nu^{\pm\mathbb R})$ &  & $\q$ &  & $\q$ &  & $\q$ &  & $\q$ &  & $\q(\nu^{\pm\mathbb R})$ \\[5pt]
     & $\q(i\sqrt 7)$ &  & $\q(i\sqrt 7)$ &  &  &  & $\q$ &  & $\q$ &  \\[5pt]
     &  &  & &  & $\q(\nu^{\pm\mathbb R})$ &  &  &  &  &
   \end{tabular}
  }
  \caption{Fields of definition for the partial duals of $\map^-$, and hence of $\map^+$. All the fields of definition are also fields of moduli.}\label{table: fields}
\end{table}

We note that every partial dual of $\map^\mathbb{R}$ is defined over the real numbers.

\begin{table}[ht]
  \centering
  \scalebox{1}{
   \begin{tabular}{ccccccccccc}
     &  &  &  &  & $\q(\nu^{\pm\mathbb R})$ &  &  &  &  & \\[5pt]
     & $\q$ &  & $\q$ &  &  &  & $\q(\sqrt{105})$ &  & $\q(\sqrt{105})$ &  \\[5pt]
     $\q(\nu^{\pm\mathbb R})$ &  & $\q$ &  & $\q$ &  & $\q$ &  & $\q$ &  & $\q(\nu^{\pm\mathbb R})$ \\[5pt]
     & $\q(\sqrt{105})$ &  & $\q(\sqrt{105})$ &  &  &  & $\q$ &  & $\q$ &  \\[5pt]
     &  &  & &  & $\q(\nu^{\pm\mathbb R})$ &  &  &  &  &
   \end{tabular}
  }
  \caption{Fields of definition for the partial duals of $\map^\mathbb{R}$.}\label{table: fields_r}
\end{table}

\begin{rmrk}Although we have noted that there are no two partial duals of $\map^\mathbb{R}$ which belong to the same Galois orbit, the partial duals $\partial_3\map^\mathbb{R}$, $\partial_4\map^\mathbb{R}$ and $\partial_{123}\map^\mathbb{R}$, $\partial_{124}\map^\mathbb{R}$ belong to Galois orbits of order 2 since they are defined over a quadratic field. Their images under the automorphism $\pm\sqrt{105}\mapsto\mp\sqrt{105}$ can be found among the partial duals of the map in figure \ref{figure: appendix105}.
\begin{figure}[ht]
  \centering
  \includegraphics[scale=.4,trim={6cm 14cm 6cm 4.75cm}]{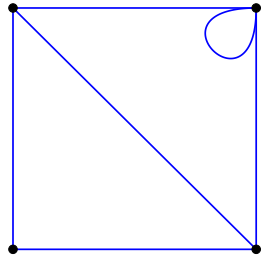}
  \caption{The map whose partial duals are conjugate to $\partial_3\map^\mathbb{R}$, $\partial_4\map^\mathbb{R}$, $\partial_{123}\map^\mathbb{R}$ and $\partial_{124}\map^\mathbb{R}$.}\label{figure: appendix105}
\end{figure}
\end{rmrk}

\vspace{\topsep}

\noindent\emph{\textbf{Overview of cartographic groups.}}

\begin{table}[ht]
  \centering
  \scalebox{.76}{
   \begin{tabular}{ccccccccccc}
     &  &  &  &  & $S_8$ &  &  &  &  & \\[5pt]
     & $A_8$ &  & $\text{PSL}(2,7)$ &  &  &  & $(1344,11686)$ &  & $(1344,11686)$ &  \\[5pt]
     $S_8$ &  & $S_8$ &  & $(1152,157849)$ &  & $(1152,157849)$ &  & $S_8$ &  & $S_8$ \\[5pt]
     & $(1344,11686)$ &  & $(1344,11686)$ &  &  &  & $\text{PSL}(2,7)$ &  & $A_8$ &  \\[5pt]
     &  &  & &  & $S_8$ &  &  &  &  &
   \end{tabular}
  }
  \caption{Cartographic groups for the partial duals of $\map^-$, and hence of $\map^+$.}\label{table: monodromy}
\end{table}

\begin{table}[ht]
  \centering
  \scalebox{1}{
   \begin{tabular}{ccccccccccc}
     &  &  &  &  & $S_8$ &  &  &  &  & \\[5pt]
     & $(288,1025)$ &  & $(288,1025)$ &  &  &  & $A_8$ &  & $A_8$ &  \\[5pt]
     $S_8$ &  & $S_8$ &  & $S_8$ &  & $S_8$ &  & $S_8$ &  & $S_8$ \\[5pt]
     & $A_8$ &  & $A_8$ &  &  &  & $(288,1025)$ &  & $(288,1025)$ &  \\[5pt]
     &  &  & &  & $S_8$ &  &  &  &  &
   \end{tabular}
  }
  \caption{Cartographic groups for the partial duals of $\map^\mathbb{R}$.}\label{table: monodromy_r}
\end{table}
\end{example}
\newpage

\section{Concluding remarks and further research}\label{section:concluding}

Throughout this paper we have seen that maps on surfaces encode various combinatorial, algebraic, topological and group and Galois-theoretic data. The action of $BC_n$ on maps is well-behaved with respect to all but the Galois-theoretic data, and at the moment it doesn't seem likely that this action will be illuminating in the theory of dessins d'enfants.

A natural step is to consider operations akin to partial duality on representations of symplectic matroids, i.e.\ the combinatorial data of linearly independent vectors of $k$-dimensional isotropic subspaces of a symplectic $2n$-space, with $k<n$. However, it is not known if general symplectic matroids satisfy a basis exchange axiom; the definition of a base of a symplectic matroid is given by a certain \emph{maximality property} with respect to a Gale order induced by a permutation in $BC_n$. An even more ambitious goal is to understand such operations in the general framework of Coxeter matroids \cite{borovik_gelfand_white} of which ordinary, symplectic and Lagrangian matroids are a special case.

Furthermore, (ordinary) matroid polytopes are closely related to cluster algebras of Grassmannians (see for example chapter 5 in \cite{postnikov07}), and it would definitely be worthwhile to investigate the role of Lagrangian (and Coxeter matroids in general) matroids in the theory of cluster algebras.

Recent work by Lando and Zhukov \cite{lando_zhukov} connecting Lagrangian matroids to knot theory via Vassiliev invariants seems as a fertile ground for further application of Lagrangian matroid theory to low-dimensional topology.

\bibliography{references}

\begin{thebibliography}{18}
\providecommand{\natexlab}[1]{#1}
\providecommand{\url}[1]{\texttt{#1}}
\expandafter\ifx\csname urlstyle\endcsname\relax
  \providecommand{\doi}[1]{doi: #1}\else
  \providecommand{\doi}{doi: \begingroup \urlstyle{rm}\Url}\fi

\bibitem[Adrianov et~al.(2009)Adrianov, Amburg, Dremov, Kochetkov, Kreines,
  Levitskaya, Nasretdinova, and Shabat]{adrianov09}
N.~M. Adrianov, N.~Ya. Amburg, V.~A. Dremov, Yu.~Yu. Kochetkov, E.~M. Kreines,
  Yu.~A. Levitskaya, V.~F. Nasretdinova, and G.~B. Shabat.
\newblock {Catalog of Dessins d'Enfants with no more than 4 edges}.
\newblock \emph{Journal of Mathematical Sciences}, 158\penalty0 (1):\penalty0
  22--80, 2009.

\bibitem[Bely\u{\i}(1980)]{belyi80}
G.~V. Bely\u{\i}.
\newblock {On Galois Extensions of a Maximal Cyclotomic Field}.
\newblock \emph{Math. USSR Izvestija}, 14:\penalty0 247--256, 1980.

\bibitem[Bely\u{\i}(2002)]{belyi02}
G.~V. Bely\u{\i}.
\newblock {A New Proof of the Three Point Theorem}.
\newblock \emph{Sb. Math.}, 193(3-4):\penalty0 329--332, 2002.

\bibitem[Birch(1994)]{birch1994noncongruence}
Bryan Birch.
\newblock Noncongruence subgroups, covers and drawings.
\newblock \emph{The Grothendieck theory of dessins d�enfants}, 200:\penalty0
  25--46, 1994.

\bibitem[Booth et~al.(2005)Booth, Borovik, and White]{Booth20051023}
Richard~F. Booth, Alexandre~V. Borovik, and Neil White.
\newblock Lagrangian pairs and lagrangian orthogonal matroids.
\newblock \emph{European Journal of Combinatorics}, 26\penalty0 (7):\penalty0
  1023 -- 1032, 2005.

\bibitem[Borovik et~al.(2003)Borovik, Gelfand, and
  White]{borovik_gelfand_white}
A.V. Borovik, I.M. Gelfand, and N.~White.
\newblock \emph{Coxeter Matroids}.
\newblock Birk\-h\"{a}us\-er, 2003.

\bibitem[Bouchet(1987)]{bouchet87}
A.~Bouchet.
\newblock Greedy algorithm and symmetric matroids.
\newblock \emph{Mathematical Programming}, 38:\penalty0 147--159, 1987.

\bibitem[Chmutov(2009)]{chmutov09}
S.~Chmutov.
\newblock {Generalized duality for graphs on surfaces and the signed
  Bollob\'as-Riordan polynomial}.
\newblock \emph{Journal of Combinatorial Theory, Ser. B.}, 99\penalty0
  (3):\penalty0 617--638, 2009.

\bibitem[Chmutov and Vignes-Tourneret(2014)]{chmutov14}
S.~Chmutov and F.~Vignes-Tourneret.
\newblock Partial duality of hypermaps, 2014.
\newblock Preprint, arXiv:1409.0632v1 [math.CO].

\bibitem[Chun et~al.(2014)Chun, Moffatt, Noble, and Rueckriemen]{moffat14}
C.~Chun, I.~Moffatt, S.~D. Noble, and R.~Rueckriemen.
\newblock Matroids, delta-matroids and embedded graphs, 2014.
\newblock Preprint, arXiv:1403.0920v1 [math.CO].

\bibitem[Girondo and Gonz\'alez-Diez(2012)]{girondo_gonzalez-diez}
E.~Girondo and G.~Gonz\'alez-Diez.
\newblock \emph{Introduction to Compact Riemann Surfaces and Dessins
  d'Enfants}, volume~79 of \emph{London Mathematical Society Student Texts}.
\newblock Cambridge University Press, 2012.

\bibitem[Lando and Zvonkin(2004)]{lando_zvonkin}
S.~K. Lando and A.~Zvonkin.
\newblock \emph{Graphs on Surfaces and their Applications}, volume 141 of
  \emph{Encyclopaedia of Mathematical Sciences}.
\newblock Springer-Verlag, 2004.

\bibitem[Lando and Zhukov(2016)]{lando_zhukov}
Sergey Lando and Vyacheslav Zhukov.
\newblock Delta-matroids and vassiliev invariants, 2016.

\bibitem[Lindsay(1959)]{lindsay59}
J.~H. Lindsay, Jr.
\newblock An elementary treatment of the imbedding of a {G}raph in a surface.
\newblock \emph{Am. Math. Mon.}, 66:\penalty0 117--118, 1959.

\bibitem[Oxley(1992)]{oxley92}
James~G. Oxley.
\newblock \emph{Matroid theory}.
\newblock Oxford Science Publications. The Clarendon Press, Oxford University
  Press, 1992.

\bibitem[Postnikov et~al.(2007)Postnikov, Speyer, and Williams]{postnikov07}
Alexander Postnikov, David Speyer, and Lauren Williams.
\newblock Matching polytopes, toric geometry, and the non-negative part of the
  grassmannian, 2007.

\bibitem[Schneps(1994)]{schneps94}
L.~Schneps, editor.
\newblock \emph{The {G}rothendieck Theory of Dessins D'Enfants}.
\newblock LMS Lecture Note Series Vol.\ 200. Cambridge University Press, 1994.

\bibitem[Sijsling and Voight(2015)]{SijslingVoight}
J.~Sijsling and J.~Voight.
\newblock {On descent of marked dessins}, 2015.
\newblock Preprint, arXiv:1504.02814 [math.AG].

\end{thebibliography}

\end{document}